\documentclass[12pt,reqno]{amsart}
\usepackage{amsmath,amssymb,enumerate}
\usepackage[dvipdfmx]{graphicx}
\usepackage{color}
\usepackage{comment}
\usepackage[all]{xy}
\usepackage{bm}

\newtheorem{tm}{Theorem}[section]
\newtheorem{lm}[tm]{Lemma}

\newtheorem{df}[tm]{Definition}
\newtheorem{pr}[tm]{Proposition}

\newtheorem{ex}[tm]{Example}

\makeatletter
\newcommand{\subscripts}[3]{%
  \@mathmeasure\z@\displaystyle{#2}%
  \global\setbox\@ne\vbox to\ht\z@{}\dp\@ne\dp\z@
  \setbox\tw@\box\@ne
  \@mathmeasure4\displaystyle{\copy\tw@_{#1}}%
  \@mathmeasure6\displaystyle{{#2}_{#3}}%
  \dimen@-\wd6 \advance\dimen@\wd4 \advance\dimen@\wd\z@
  \hbox to\dimen@{}\mathop{\kern-\dimen@\box4\box6}%
}
\makeatother

\newcommand{\nn}{\nonumber}

\newcommand{\h}{\mathrm{H}}
\newcommand{\e}{\mathrm{e}}
\newcommand{\x}{\bm{\mathrm{x}}}

\newcommand{\ve}{\varepsilon}

\newcommand{\Ker}{\mathrm{Ker}\,}

\newcommand{\del}{\partial}
\newcommand{\ol}{\overline}

\newcommand{\vx}{\vec{x}}
\newcommand{\vs}{\vec{s}}
\newcommand{\vt}{\vec{t}}
\newcommand{\vsg}{\vec{\sigma}}
\newcommand{\va}{\vec{a}}
\newcommand{\vb}{\vec{b}}
\newcommand{\ii}{\mathrm{i}}
\newcommand{\la}{\langle}

\newcommand{\ra}{\rangle}

\newcommand{\vc}{\vec{c}}
\newcommand{\vk}{\vec{k}}

\allowdisplaybreaks[3]

\makeatletter
 
 \@addtoreset{equation}{section}
\makeatother

\begin{document}
\title[Random walks on crystal lattices and multiple zeta functions]{
Random walks on crystal lattices 
and multiple zeta functions
}
\author[T. Aoyama]{Takahiro Aoyama}
\author[R. Namba]{Ryuya Namba}
\date{\today}
\address[T. Aoyama]{Department of Applied Mathematics, 
Faculty of Science, 
Okayama University of Science, 
1-1 Ridaicho, Kita-ku, Okayama 700-0005, Japan}
\email{{\tt{aoyama@xmath.ous.ac.jp}}}
\address[R. Namba]{Department of Mathematics,
Faculty of Education,
Shizuoka University, 836, Ohya, Suruga-ku, Shizuoka, 
422-8529, Japan}
\email{{\tt{namba.ryuya@shizuoka.ac.jp}}}


\maketitle 
%
%
\begin{abstract}
Crystal lattices are known to be one of the generalizations of classical periodic lattices
which can be embedded into some Euclidean spaces properly.  
As to make a wide range of multidimensional discrete distributions on Euclidean spaces 
more treatable, 
multidimensional Euler products and multidimensional Shintani zeta functions 
on crystal lattices are introduced. 
They are completely different from existing 
Ihara zeta functions on graphs in that 
our zeta functions are defined on crystal lattices directly. 
Via a concept of periodic realizations of crystal lattices, 
we make it possible to provide many kinds of multidimensional discrete distributions explicitly.  
In particular, random walks on crystal lattices whose range is infinite 
and such random walks whose range is finite are constructed
by multidimensional Euler products and multidimensional Shintani zeta functions, respectively. 
We give some comprehensible examples as well. 
\end{abstract}


\section{{\bf Introduction}}

\subsection{Riemann zeta functions in probabilistic view}

Zeta functions have been investigated extensively 
in number theory and other areas among mathematics. 
The most classical one is known as the {\it Riemann zeta function} $\zeta(s)$. It is a function of a complex variable $s=\sigma+\ii t$ for $\sigma>1$ and $t\in\mathbb{R}$ given by
\begin{align}
\zeta(s):=\sum_{n=1}^\infty \frac{1}{n^s} 
=\prod_{p \in \mathbb{P}}\Big(1-\frac{1}{p^s}\Big)^{-1},\label{Euler}
\end{align}
where we denote by $\mathbb{P}$ the set of all prime numbers.
The product representation of (\ref{Euler})  is called the {\it Euler product}. 
It is known that the function $\zeta(s)$ converges absolutely in the half-plane 
$\{\sigma+\ii t \, | \, \sigma>1\}$ and uniformly
in every compact subset of the half-plane. 
We also note that the Riemann zeta function can be extended to a 
meromorphic function on the 
complex plane $\mathbb{C}$ having a single pole at $s=1$ by analytic continuation. 

On the other hand, there exists a well-known relation between 
the Riemann zeta function $\zeta(s)$ and probability distributions on
$\mathbb{R}$.  
We refer to \cite{JW}, \cite{Khinchine}, \cite{GK}
and \cite{LH01}.
Afterwards, Biane, Pitman and Yor \cite{BPY} reviewed some known results on  the Riemann zeta function in 
probabilistic view. 
See also \cite{BHNY}
for an interesting topic among the Riemann zeta function, Jacobi theta function
and some probability distributions.

Before we state the relation, we review a terminology in probability theory. 

\begin{df}[infinite divisible distribution] 
A probability measure $\mu$ on $\mathbb{R}^d$ 
is said to be infinitely divisible if, 
for any positive integer $n$, there exists a probability measure $\mu_n$ on $\mathbb{R}^d$ such that 
$\mu=\mu_n^{*n}$, where 
$\mu_n^{*n}$ is the $n$-fold convolution of $\mu_n$ itself. 
\end{df}

We denote by $I(\mathbb{R}^d)$ the set of all infinitely divisible probability measures on $\mathbb{R}^d$. 
Let $\mu$ be a probability measure on $\mathbb{R}^d$ and 
\[
\widehat{\mu}(\vt):=\int_{\mathbb{R}^d}
\e^{\ii \la \vt, \vx \ra} \, \mu(d\vx), \qquad \vt \in \mathbb{R}^d,
\]
the characteristic function of $\mu$, 
where $\la \cdot, \cdot \ra$ is 
the usual inner product on $\mathbb{R}^d$.
It is easily seen that $\mu \in I(\mathbb{R}^d)$ if and only if the $n$-th root of its characteristic function $\widehat{\mu}^{1/n}$ is also a characteristic function for every $n \in \mathbb{N}$. 
Note that this class is important when we consider probabilistic limit theorems given by sums of independent and identically distributed random variables such as the central limit theorem and the law of large numbers.

\begin{ex}[compound Poisson distribution]\normalfont
The class of {\it compound Poisson distributions} is known as one of 
the most important subclasses of infinitely divisible distributions. 
A probability distribution $\mu_{\mathrm{CP}}$ 
is compound Poisson if its characteristic function is given by 
$$
\widehat{\mu_{\mathrm{CP}}}(\vec{t})=\exp\Big( \lambda\big(\widehat{\rho}(\vec{t}) - 1\big)\Big), \qquad \vec{t} \in \mathbb{R}^d,
$$
for some $\lambda>0$ and some distribution $\rho$ with $\rho(\{0\})=0$. 
Note that the Poisson distribution is a special case when $d=1$ and $\rho=\delta_1$, where $\delta_1$ is the delta 
measure at $x=1$. 

Let $Y$ be an $\mathbb{R}^d$-valued 
random variable whose distribution is $\rho$. 
We take independent copies $Y_1, Y_2, \dots$ of $Y$ and a Poisson random variable $K$ with 
a parameter $\lambda>0$ independent of $Y$. Then the characteristic function $\widehat{\mu_Z}$ of a random variable 
$Z:=Y_1+Y_2+\cdots+Y_K$ is given by
$$
\begin{aligned}
\widehat{\mu_Z}(\vec{t})&=\mathbf{E}\big[\e^{\ii \la \vec{t}, Y_1+Y_2+\cdots+Y_K\ra}\big]
=\sum_{n=1}^\infty \mathbf{E}\big[\e^{\ii \la \vec{t}, Y_1+Y_2+\cdots+Y_n\ra} :  K=n\big]\\
&=\sum_{n=1}^\infty \widehat{\rho}(\vec{t})^n \cdot \Big(\e^{-\lambda} \frac{\lambda^n}{n!}\Big)
=\exp\Big( \lambda\big(\widehat{\rho}(\vec{t}) - 1\big)\Big), \qquad 
\vec{t} \in \mathbb{R}^d,
\end{aligned}
$$
which means that the distribution of $Z$ is of compound Poisson. 

\end{ex}

The following is well-known. 

\begin{pr}[L\'evy--Khintchine representation, cf.\,{\cite[Theorem 8.1]{Sato}}]
\label{Levy-Khintchine}
{\rm (1)} If $\mu \in I(\mathbb{R}^d)$, then we have
\begin{equation}\label{LK}
\widehat{\mu}(\vec{t})=\exp\Big(-\frac{1}{2}\la A\vec{t}, \vec{t}\ra + \ii \la \vec{\gamma}, \vec{t}\ra
+\int_{\mathbb{R}^d}\Big(\e^{\ii \la \vec{t}, \vec{x}\ra}-1-\frac{\ii \la \vec{t}, \vec{x}\ra}{1+|\vec{x}|^2}\Big) \, \nu(d\vec{x})\Big), \quad \vec{t} \in \mathbb{R}^d,
\end{equation}
where $A$ is a symmetric non-negative definite $d \times d$-matrix, 
$\nu$ is a measure on $\mathbb{R}^d$ satisfying
\begin{equation}\label{Levy measure}
\nu(\{0\})=0 \quad \text{and} \quad \int_{\mathbb{R}^d}\min\{|\vx|^2, 1\} \, \nu(d\vx)<+\infty
\end{equation}
and $\vec{\gamma} \in \mathbb{R}^d$. 

\vspace{2mm}
\noindent
{\rm (2)} The representation of
$\widehat{\mu}$ in \eqref{LK} 
by $A$, $\nu$ and $\vec{\gamma}$ 
is unique. 

\vspace{2mm}
\noindent
{\rm (3)} Conversely, if $A$ is a symmetric non-negative definite $d \times d$-matrix, 
$\nu$ is a measure on $\mathbb{R}^d$ satisfying \eqref{Levy measure} and $\vec{\gamma} \in \mathbb{R}^d$, then there exists $\mu \in I(\mathbb{R}^d)$ whose characteristic function
is given by \eqref{LK}. 
\end{pr}

The measure 
$\nu$ and the triplet $(A, \nu, \vec{\gamma})$ are called a L\'evy measure  and 
a L\'evy--Khintchine triplet of $\mu \in I(\mathbb{R}^d)$, respectively. 
There is another form of \eqref{LK} if the L\'evy measure $\nu$ satisfies an additional condition.

\begin{pr}[cf.\,{\cite[Section 8]{Sato}}]\label{LK-formula}
In Proposition {\rm \ref{Levy-Khintchine}}, 
if  the L\'evy measure $\nu$ in \eqref{LK} satisfies $\int_{|\vx|<1}|\vx| \, \nu(d\vx)<+\infty$, 
then we can rewrite the representation \eqref{LK} by 
\begin{equation}\label{LK2}
\widehat{\mu}(\vec{t})=\exp\Big(-\frac{1}{2}\la A\vec{t}, \vec{t}\ra + \ii \la \vec{\gamma}_0, \vec{t}\ra
+\int_{\mathbb{R}^d}\big(\e^{\ii \la \vec{t}, \vec{x}\ra}-1\big) \, \nu(d\vx)\Big), \qquad \vec{t} \in \mathbb{R}^d,
\end{equation}
where $\vec{\gamma}_0:=\vec{\gamma} - \int_{|\vx|<1}\vx(1+|\vx|^2)^{-1} \, \nu(d\vx)$. 
\end{pr}

Let us go back to the Riemann zeta function $\zeta(s)$.  
Next we introduce a class of probability distributions on $\mathbb{R}$ generated by $\zeta(s)$. 

\begin{df}[Riemann zeta distribution on $\mathbb{R}$, cf.\,\cite{GK}]
Fix $\sigma>1$. A probability distibution $\mu_\sigma$ 
is called a Riemann zeta distribution if 
\begin{equation}\label{Riemann zeta dist}
\mu_\sigma(\{-\log n\})=\frac{n^{-\sigma}}{\zeta(\sigma)}, \qquad n \in \mathbb{N}. 
\end{equation}
\end{df}

It is readily verified that the function given by
$$
f_{\sigma}(t):=\frac{\zeta(\sigma+\ii t)}{\zeta(\sigma)}, \qquad t \in \mathbb{R},
$$
coincides with the characteristic function of $\mu_\sigma$ (see e.g., \cite{GK}).  
Moreover,  the Riemann zeta distribution is known to be  infinitely divisible. 

\begin{pr}[cf.\,\cite{GK}]
Let $\mu_{\sigma}$ be a Riemann zeta distribution on $\mathbb{R}$.
Then, $\mu_{\sigma}$ is of compound Poisson on 
$\mathbb{R}$ and
\begin{align}
\log f_{\sigma}(t)=
\sum_{p \in \mathbb{P}}\sum_{r=1}^{\infty}\frac{p^{-r\sigma}}{r}\left(\e^{-{\rm i}rt\log p}-1\right)
=\int_0^{\infty}\left(\e^{-{\rm i}tx}-1\right)N_{\sigma}(dx), \qquad 
t \in \mathbb{R},
\label{zeta CF}
\end{align}
where $N_{\sigma}$ is a finite L\'evy measure on $\mathbb{R}$ given by
\begin{align*}
N_{\sigma}(dx)=\sum_{p \in \mathbb{P}}\sum_{r=1}^{\infty}\frac{p^{-r\sigma}}{r}\delta_{r\log p}(dx).
\end{align*}
\end{pr}

We can also explain Equation \eqref{zeta CF} in terms of random variables. 
Fix $\sigma>1$. 
We define an $\mathbb{R}$-valued random variable $Y_\sigma$ given by
$$
\mathbf{P}(Y_\sigma=-r\log p)=\frac{1}{N_\sigma(\mathbb{R})}\cdot\frac{p^{-r\sigma}}{r}, 
\qquad r \in \mathbb{N}, \, p \in \mathbb{P}. 
$$
Let $Y_\sigma^1, Y_\sigma^2, \dots$ be independent copies of $Y_\sigma$ and 
$K_\sigma$ be a Poisson random variable with a parameter 
$N_\sigma(\mathbb{R})>0$ independent of $Y_\sigma$. 
Then we see that the probability distribution of a random variable 
$Z_\sigma:=Y_\sigma^1+Y_\sigma^2+\cdots+Y_\sigma^{K_\sigma}$ 
coincides with $\mu_\sigma$ given by \eqref{Riemann zeta dist}. 


\subsection{Multidimensional zeta functions in probabilistic view}

As mentioned in the previous subsection, 
the Riemann zeta distribution is one of the examples of 
discrete probability distributions on $\mathbb{R}$
with infinitely many mass points. 
However, only a few examples of such discrete probability distributions are known explicitly
when we look at higher dimension cases. 
In order to treat such discrete distributions extensively, 
{\it multidimensional Shintani zeta functions} are introduced in view of 
series representations 
and investigated some useful relations with probability theory in \cite{AN2}.

For $\vec{c} \in \mathbb{R}^d$ and 
$\vs \in \mathbb{C}^d$, we write $\langle \vec{c} ,\vs \rangle := \langle \vec{c} ,\vec{\sigma}\rangle + 
{\rm i}\langle\vec{c} ,\vt \rangle$, 
where $\vec{\sigma}, \vt \in \mathbb{R}^d$ 
and $\vs =\vec{\sigma} +{\rm i}\vt$. 

\begin{df}[multidimensional Shintani zeta function, cf.\,\cite{AN2}]
\label{Def:Shintani zeta}
Let $d, m, r \in \mathbb{N}$, $\vs \in \mathbb{C}^d$ and 
$(n_1, n_2, \dots, n_r) \in \mathbb{Z}_{\ge 0}^r$. 
For $\lambda_{\ell j}, u_j>0$, $\vec{c}_\ell \in \mathbb{R}^d$, where 
$j=1, 2, \dots, r$ and $\ell=1, 2, \dots, m$, 
and a $\mathbb{C}$-valued function
$\theta(n_1, n_2, \dots, n_r)$ satisfying
$|\theta(n_1, n_2, \dots, n_r)|=O\big((n_1+n_2+\cdots+n_r)^\ve\big)$, for any $\ve>0$,  
we define a multidimensional Shintani zeta function $Z_S(\vs)$ by
\begin{equation}\label{Shintani zeta}
Z_S(\vs):=\sum_{n_1, n_2, \dots, n_r=0}^\infty \frac{\theta(n_1, n_2, \dots, n_r)}
{\prod_{\ell=1}^{m}\big( \sum_{j=1}^r \lambda_{\ell j}(n_j+u_j)\big)^{\la \vec{c}_\ell, \vs \ra}}.
\end{equation}
\end{df}

We write $\mathcal{Z}_S^{(d)}$ for the class consisting of all $d$-dimensional
Shintani zeta functions of the form \eqref{Shintani zeta}.  
Put 
$\vs:=\vec{\sigma} +{\rm i}\vt, \, \vec{\sigma}, \vt \in \mathbb{R}^d.$
The absolute convergence of $Z_S (\vs)$ is given as follows:

\begin{pr}[cf.\,{\cite[Theorem 1]{AN2}}]
\label{Shintani zeta absolutely convergence}
The series defined by 
$\eqref{Shintani zeta}$ converges absolutely in the region 
$\min \{\mathrm{Re}\,\langle \vec{c}_{\ell}, \vs\rangle |\,
\ell=1, 2, \dots m\}>r/m$. 
\end{pr}

Let $\theta(n_1, n_2, \dots, n_r), \, 
(n_1, n_2, \dots, n_r) \in  \mathbb{Z}^r_{\ge 0}$,
be a nonnegative or nonpositive definite function and $\vec{\sigma} \in \mathbb{R}^d$ a vector satisfying
$\min\{\mathrm{Re}\,\langle \vec{c}_{\ell}, \vec{\sigma}\rangle|\,\ell=1, 2, \dots, m\} >r/m$. 
Then, 
as an extension of the Riemann zeta case, 
it is shown in \cite[Theorem 3]{AN2} that Shintani zeta functions also generate characteristic functions of probability distributions on $\mathbb{R}^h, 1\le h\le d,$
by putting 
$$
f_{\vsg}(\vt):=\frac{Z_S(\vsg+\ii \vt)}{Z_S(\vsg)}, \qquad \vt \in \mathbb{R}^d. 
$$
However, there exists no effective methods to know their infinite divisibilities. 
As to show that, 
they introduced a new multidimensional polynomial  
Euler products in \cite{AN3}


\begin{df}[multidimensional polynomial Euler product, 
cf.\,\cite{AN3}]
Let $d, \, m \in \mathbb{N}$ and $\vs \in \mathbb{C}^d$. 
For $-1 \le \alpha_{\ell}(p) \le 1$ and 
$\va_{\ell} \in \mathbb{R}^d$, 
$\ell=1, 2, \dots, m$ and $p \in \mathbb{P}$, 
we define a $d$-dimensional polynomial Euler product $Z_E(\vs)$ given by
\begin{equation}\label{ZE}
Z_E(\vs):=\prod_{p \in \mathbb{P}}\prod_{\ell=1}^{m} \big( 1 - \alpha_{\ell}(p)
p^{-\la \va_{\ell}, \vs \ra}\big)^{-1}.
\end{equation}
\end{df}

We write $\mathcal{Z}_E^{(d)}$ for the class consisting of all $d$-dimensional polynomial
Euler product of the form \eqref{ZE}. 
Note that  
the multidimensional polynomial Euler products were generalized to complex coefficients cases in \cite{Nakamura}. 
We can see that the product \eqref{ZE} also converges absolutely.

\begin{pr}[cf.\,{\cite[Theorem 2.3]{AN3}}]
$Z_E(\vs)$ converges absolutely and has no zeros in the region $\min\{\mathrm{Re}\, \langle \va_\ell,\vs\rangle|\,\ell=1, 2, \dots, m\} >1$. 
\end{pr}

Let $\vec{\sigma} \in \mathbb{R}^d$ be a vector satisfying
$\min\{\mathrm{Re}\, \langle \va_\ell,\vsg\rangle|\,\ell=1, 2, \dots, m\} >1$. 
We now put 
$$
f_{\vsg}(\vt):=\frac{Z_E(\vsg+\ii \vt)}{Z_E(\vsg)}, \qquad \vt \in \mathbb{R}^d.
$$
Then we may expect that the function $f_{\vsg}(\vt)$ is also to be a characteristic function
as well as the case of Shintani zeta functions. 
However, it is not trivial that $f_{\vsg}(\vt)$ to be a characteristic function of some probability distribution on $\mathbb{R}^d$. 
Therefore, it is natural to ask when $f_{\vsg}(\vt)$ to be so. 
To find that out, the following two conditions for $\mathbb{R}^d$-valued vectors 
$\va_{\ell}, \, \ell=1, 2, \dots, m$,  were introduced in \cite{AN3}. 

\begin{itemize}
\item[{\bf (LI)}:] $\va_{\ell }, \, \ell=1, 2, \dots, m$, are {\it linearly independent}.
\vspace{2mm}
\item[{\bf (LR)}:] $\va_{\ell }, \, \ell=1, 2, \dots, m$, are linearly dependent but {\it linearly independent over the rationals}. 
Namely, for some $\va \in \mathbb{R}^d$, 
it holds that $\va_{\ell }=\psi_{\ell }\va, \, \ell=1, 2, \dots, m$, where each $\psi_{\ell }$
are algebraic real numbers and linearly independent over the rationals.
\end{itemize}
Under either of these two conditions, a necessary and sufficient condition for 
$f_{\vsg}$ to be a characteristic function was obtained 
as in the following.

\begin{pr}[cf.\,{\cite[Theorem 3.8]{AN3}}] \label{AN-zeta}
Suppose that 
$\va_{\ell}, \, \ell=1, 2, \dots, m$, 
in {\rm (\ref{ZE})} satisfy either {\bf (LI)} or {\bf (LR)}. 
Then $f_{\vsg}(\vt)$ is a 
characteristic function if and only if 
$\alpha_{\ell}(p) \geq 0$
for all $\ell=1, 2, \dots, m$ and  $p \in \mathbb{P}$. 
Moreover, $f_{\vsg}(\vt)$ is a compound Poisson characteristic function with
its finite L\'evy measure $N_{\vsg}(d\vx)$ 
on $\mathbb{R}^d$ given by
$$
N_{\vsg}(d\vx)=\sum_{p \in \mathbb{P}}
\sum_{r=1}^\infty \sum_{\ell=1}^m
 \frac{1}{r}\alpha_{\ell}(p)^r p^{-r\la \va_{\ell}, \vsg\ra}
 \delta_{(r\log p) \va_{\ell}}(d\vx).
$$
\end{pr}
\noindent

\subsection{The purpose of the paper}

Random walks on graphs are one of the first contacts in graph theory and probability theory.
We often regard some good, periodic and so on, graphs as crystal lattices in Euclidean space by
a transformation called a realization.
By adopting simple random walks, which are allowed to move to one of the nearest neighbors 
at each steps, we have rich mathematical observations among geometry and graph theory. 
Some probabilistic limit theorems such as the central limit theorem and the large deviation principle
are shown in some cases and they let us know how nice, symmetric and so on, the graphs are.
See e.g., \cite{Sunada} and \cite{Woess} for details of such developments with extensive references therein. 

Though simple random walks are good enough for us to study comprehensive graphs, 
we need more general random walks to understand properties of complicated ones.
There were some difficulties to do such things for the lack of existences of 
treatable multivariate functions which are the Fourier transformations (i.e., 
characteristic functions) of multidimensional discrete distributions with finite and 
especially infinite supports.
To break the walls, Aoyama and Nakamura have introduced some new multiple series and infinite 
products which are called multidimensional Shintani zeta functions and 
multidimensional polynomial Euler products as mentioned above, respectively.
They made it useful for us to describe some linearly or periodically supported discrete distributions
in Euclidean spaces by applying analytic number theory.
In the present paper, we try to give a new contact of graphs and general random walks by 
using their multiple zeta functions to break some walls left.

Here, we should mention that there are some existing studies 
on zeta functions on graphs. 
The study of zeta functions on graphs dates back to 1960s. 
It was originated in \cite{Ihara} considering a $p$-adic analogue of Selberg zeta functions, 
which are now called Ihara zeta functions.
The Ihara zeta function is defined by 
a kind of the Euler product on the set of so-called
``prime cycles'' in finite graphs. 
See e.g., \cite{KS000} for the definition and some properties of the function. 
It is known that there are explicit relations
with the Riemann and some generalized zeta functions appeared in number theory. 
Several interesting aspects of the Ihara zeta function related to e.g. spectral geometry on finite graphs and random matrix theory
with extensive references can be found in \cite{Terras}.  
Some authors have considered extensions of Ihara zeta functions to some
infinite graph cases 
by applying operator-algebraic approaches. 
See e.g., \cite{GIL} and \cite{LPS} for more details
and references therein. 

On the other hand, we here emphasize that 
such Ihara zeta functions may not be useful
when we try to capture probabilistic objects 
such as random walks and even treatable probability distributions
on infinite graphs. 
In fact, it is difficult to capture random 
trajectories on graphs via the Ihara zeta function,
since each factor of the function is given in terms of a certain coset of ``cycles'' in the graphs.
Furthermore, the study of Ihara zeta functions has been developed mainly in the context of geometry of graphs. 
On the contrary, the multiple zeta functions introduced in the present paper  are completely different from Ihara zeta functions.
In view of some known observations among classical zeta functions and probabilistic objects, we believe that 
more fruitful contributions to random walks on graphs 
can be obtained by making use of such multiple zeta functions.

The rest of the present paper is organized as follows: 
We review some basic terminologies from graph theory and introduce
notions of crystal lattices and their periodic realizations into 
some Euclidean spaces in Section 2. 
We introduce multidimensional Shintani zeta functions on crystal lattices in Section 3. 
We see that some finite range random walks on crystal lattices are defined in terms of the multidimensional Shintani zeta functions.
As a counterpart of Section 3, we consider 
multidimensional polynomial Euler products on crystal lattices 
and study a certain subclass of them in Section 4.
We also see that some infinite range random walks on crystal lattices are defined in terms of the 
multidimensional ``finite'' Euler products. 
Several comprehensible examples of crystal lattices 
of dimension 1 and 2 and multiple zeta functions on them 
are given as well in Section 5.

\section{{\bf Crystal lattices and their periodic realizations}}

\subsection{Covering graphs}
There exist a lot of classes of infinite graphs which possess geometric 
features such as periodicities, volume growths and so on. 
Crystal lattices are known as one of the most typical classes of periodic graphs
and have been well-studied from geometric perspectives. 
Such graphs are regarded as discrete analogues of covering spaces of compact manifolds. 
In particular, their periodicities are clearly described in terms of covering transformation groups.   
For more details, we refer to \cite{KS06}
and \cite{Sunada}.

Let $X=(V, E)$ be an oriented and connected graph, 
where $V$ is the set of all vertices and $E$ is the set of all oriented edges.  
For an oriented edge $e \in E$, 
we denote by $o(e)$ and $t(e)$ the {\it origin} and the {\it terminus} of $e$, respectively. 
The {\it inverse edge} of $e \in E$ is  an edge $\ol{e} \in E$ satisfying $o(\ol{e})=t(e)$
and $t(\ol{e})=o(e)$. 
Note that our graphs possibly have {\it loops} ($e \in E$ with $o(e)=t(e)$)
and {\it multiple edges} ($e_1, e_2 \in E$ with $e_1 \neq e_2$, $o(e_1)=o(e_2)$ and $t(e_1)=t(e_2)$). 
A {\it path} $c$ in $X$ of length $n$ is a sequence $c=(e_1, e_2, \dots, e_n)$ of $n$ edges 
$e_1, e_2, \dots, e_n \in E$ with $o(e_{i+1})=t(e_i)$ for $i=1, 2, \dots, n-1$.
We denote by $\Omega_{x, n}(X)$ 
the set of all paths in $X$ of length $n \in \mathbb{N} \cup \{\infty\}$ starting from $x \in V$. 

We write $E_x$ for the set of all edges whose origin is $x \in V$, 
that is, 
$$
E_x=\{e \in E \, | \, o(e)=x\}, \qquad
x \in V. 
$$
Throughout the present paper, we always consider 
{\it locally finite} graph, that is, 
$\deg(x):=|E_x|<\infty$ for $x \in V.$

The group of all automorphisms on a set $M$ is denoted by $\mathrm{Aut}(M)$. 
The notion of group actions on graphs is stated as in the following.

\begin{df}[group action]
{\rm (1)} We say that {\it a group $G$ acts on a graph $X=(V, E)$} if two
homomorphisms $\phi_V : G \to \mathrm{Aut}(V)$ and $\phi_E : G \to \mathrm{Aut}(E)$
are given and they satisfy that 
$$
o\big(\phi_E(g)e\big)=\phi_V(g)o(e), \qquad t\big(\phi_E(g)e\big)=\phi_V(g)t(e)
$$
and there is no edges with $\phi_E(g)e=\ol{e}$ for $g \in G$. 
We denote the actions of $g \in G$ on $x \in V$ and $e \in E$
by $gx$ and $ge$, respectively. 

\vspace{2mm}
\noindent
{\rm (2)} 
We say that a group $G$ acts on $X$ {\it freely}  if
$gx=hx$ implies $g=h$ for every $x \in V$. 
\end{df}

Let $X=(V, E)$ be a graph and suppose that a group $G$ acts on $X$ freely. 
We put $V_0:=G \backslash V$ and $E_0:=G \backslash E$, the orbits for the $G$-action. 
Then we can find a unique graph structure $X_0:=(V_0, E_0)$ with a 
 canonical projections $\pi_V : V \to V_0$ and $\pi_E : E \to E_0$. 
The graph $X_0$ is called a {\it quotient graph} of $X$ by the $G$-action
and it is denoted by $G \backslash X$. 

We here introduce the notion of covering graphs as a geometric analogue
of Galois theory in abstract algebra. 
Let $X_0=(V_0, E_0)$ be a connected and finite graph. 
Then the concept of covering map is defined as follows:

\begin{df}[covering map and covering graph]
A morphism $\pi : X \to X_0$ is said to be a {\it covering map} if
$\pi : V \to V_0$ is surjective and the restriction map
$\pi|_{E_x} : E_x \to (E_{0})_{\pi(x)}$ is bijective for every $x \in V$. 
We call $X$ a covering graph of $X_0$. 
\end{df}

Namely, the covering map $\pi$ preserves the information of local 
relations between vertices and edges. 
The following is the definition of covering transformation group 
corresponding to covering graphs. 

\begin{df}[covering transformation group]
Let $X$ be a covering graph of a finite graph $X_0$ and $\pi : X \to X_0$ a covering map.
An automorphism $\gamma$ of $X$ is called a covering transformation if $\pi \circ \gamma=\gamma$. 
The set of all covering transformations of $X$ forms a group $\Gamma$ under the
composition of maps, which is called a covering transformation group of $X$. 
\end{df}
\noindent

It is easily seen that the covering transformation group $\Gamma$ acts on 
the graph $X$ {\it freely}. Namely, if there is a vertex $x \in V$ satisfying
$\gamma x=x$, then $\gamma=\bm{1}_{\Gamma}$ follows, where 
$\bm{1}_{\Gamma}$ stands for the unit of $\Gamma$. 
Moreover, we assume that the covering map $\pi : X \to X_0$
is {\it regular} throughout the present paper, that is, 
the free action of $\Gamma$ on every fiber $\pi^{-1}(x), \, x \in V,$ is transitive. 
We then know that the quotient graph $\Gamma \backslash X$ is isomorphic to $X_0$,
so that the canonical surjection $\pi : X \to \Gamma \backslash X$ is a regular covering map whose covering transformation group is $\Gamma$. 

Our focus of interest is the following.

\begin{df}[crystal lattice]
An infinite covering graph $X = (V, E)$  is called a crystal lattice
if the covering transformation group $\Gamma$ is finitely generated and abelian. 
\end{df}

We occasionally call $d:=\mathrm{rank}\,\Gamma$ the dimension of $X$, which is denoted by $\dim X$. 
We also assume that 
$\Gamma$ 
has {\it no torsions}. 
Namely, if $\gamma \in \Gamma$ satisfies $\gamma^n=\bm{1}_\Gamma$ for some
$n \in \mathbb{N}$, then $\gamma=\bm{1}_\Gamma$. 
In this case, we have $\Gamma \cong \mathbb{Z}^d$ for some   integer $d \geq 1$.

\subsection{Homology groups of finite graphs}

The notion of homology groups allows us to treat topological spaces in 
an algebraic point of view. It has been well-studied intensively and extensively
by both algebraists and geometers. Generally speaking, 
it is  difficult to compute homology groups. However, 
graph cases are known to be much easier ones to compute them. 
The definition and several properties of homology groups of finite graphs
are given in this subsection. 

Let $X_0=(V_0, E_0)$ be a finite graph.  
We define the 0-{\it chain group} and 1-{\it chain group} of $X_0$ by
$$
C_0(X_0, \mathbb{Z}):=\Big\{ \sum_{x \in V_0}a_x x \, \Big| \, a_x \in \mathbb{Z}\Big\},
\qquad C_1(X_0, \mathbb{Z}):=
\Big\{ \sum_{e \in E_0}a_e e \, \Big| \, a_e \in \mathbb{Z}, \, \ol{e}=-e\Big\},
$$
respectively. 
The boundary operator $\del : C_1(X_0, \mathbb{Z}) \to C_0(X_0, \mathbb{Z})$ 
is defined by the homomorphism satisfying $\del(e)=t(e)-o(e)$ for $e \in E_0$.
Note that $\del(\ol{e})=-\del(e)$ for $e \in E_0$ due to $\ol{e}=-e$. 
Then   the first homology group is defined as in the following. 

\begin{df}[first homology group]
The first homology group of $X_0$ is defined by 
$$
\h_1(X_0, \mathbb{Z}):=\Ker(\del) \subset C_1(X_0, \mathbb{Z}).
$$ 
\end{df}

\noindent
An element of $\h_1(X_0, \mathbb{Z})$ is usually called a 1-{\it cycle}. 
By definition, we see that $\h_1(X_0, \mathbb{Z})$ is a free abelian group of finite rank. 
We call 
$$
b_1(X_0):=\mathrm{rank}\,\h_1(X_0, \mathbb{Z})
$$ 
the {\it first Betti number} of $X_0$, which indicates the number of ``holes'' of $X_0$. 
Indeed, we easily obtain 
$
b_1(X_0)=|E_0|/2-|V_0|+1. 
$

There is an important relation between $\h_1(X_0, \mathbb{Z})$ and closed 
paths in $X_0$.  
For a path $c=(e_1, e_2, \dots, e_n)$ in $X_0$, 
we denote by $\la c \ra$ the 1-chain $e_1+e_2+\cdots+e_n$. 
Then it is readily verified that $\del(\la c \ra)=0$ if $c$ is a closed path, 
which means that $ \la c \ra \in \h_1(X_0, \mathbb{Z})$.  
Conversely, every 1-cycle $\alpha \in \h_1(X_0, \mathbb{Z})$ 
is represented by a closed path (see \cite[page 40]{Sunada}). 

\if 
The notion of fundamental groups is also well-known 
as another algebraic tool to treat topological spaces. 
Fundamental groups are usually far from abelian groups, while
homology groups are always abelian. 
However, the notion also gives us rich information about holes of topological
spaces. 

To define it, we start with homotopy classes of paths in a graph $X=(V, E)$.
Let $x, y \in V$. We write 
$\mathcal{C}(x, y)$ for the set of all paths $c$ with $o(c)=x$ and $t(c)=y$. 
It is easily seen that $\mathcal{C}(x, x)$ means the set of loops with base point $x$. 
Two paths $c_1$ and $c_2$ in $\mathcal{C}(x, y)$ are said to be {\it homotopic},
say $c_1 \sim c_2$, 
if there exists a finite sequence of paths 
$\{c^{(1)}, c^{(2)}, \dots, c^{(k)}\} \subset \mathcal{C}(x, y)$
such that $c_1=c^{(1)}$, $c_2=c^{(k)}$ and each $c^{(i+1)}$ is obtained by removing
 {\it back-tracking edges} of $c^{(i)}$ for $i=1, 2, \dots, k-1$. 
 It is clear that the relation $\sim$ is an equivalence relation on $\mathcal{C}(x, y)$.
Equivalence classes of $\sim$ are called {\it homotopy classes}. 
Then the definition of fundamental groups is given as follows:

\begin{df}
{\rm (fundamental group)}
Let $X=(V, E)$ be a graph and fix $x \in V$. We put
$\pi_1(X, x):=\mathcal{C}(x, x)/\sim. $
We define multiplication of two homotopy classes $[c_1]$ and $[c_2]$
by $[c_1]*[c_2]:=[c_1 \cdot c_2]$. Then $(\pi_1(X, x), *)$ becomes a group, 
which is called the fundamental group of $X$ with base point $x$. 
\end{df}

\noindent
We note that the fundamental group $\pi_1(X, x)$ is independent of the choice 
of the base point $x \in V$. Therefore, we may write $\pi_1(X)$ instead of $\pi_1(X, x)$. 

Let us consider a relation between fundamental groups of covering graphs and
homology groups. 
Let $X$ be a crystal lattice with a finite graph $X_0$ whose covering transformation group is $\Gamma \cong \mathbb{Z}^d$. Then there exists a 
canonical surjective homomorphism $\rho : \pi_1(X) \to \Gamma$. 
The abelianization of $\rho$ gives a surjective homomorphism 
$
\rho^{\text{ab}} : \pi_1(X)/[\pi_1(X), \pi_1(X)] \cong \h_1(X_0, \mathbb{Z}) \to \Gamma. 
$
This map and its linearization 
$$
\rho_{\mathbb{R}} : \h_1(X_0, \mathbb{Z}) \otimes \mathbb{R} \cong \h_1(X_0, \mathbb{R}) \to \Gamma \otimes \mathbb{R} \cong \mathbb{R}^d
$$
plays an important role in considering periodic realizations of abelian covering graphs 
into $\mathbb{R}^d$ in next subsection.
Before closing this subsection,  the notion of the maximality of 
abelian covering graphs is given. 
\fi 

\begin{df}[maximal abelian covering graph] 
A crystal lattice $X$ of a finite graph $X_0$ is said to be
maximal if the covering transformation group $\Gamma \cong \mathbb{Z}^d$ is isomorphic 
to $\h_1(X_0, \mathbb{Z})$.
\end{df}

We can also say that a crystal lattice $X$ of a finite graph $X_0$ is maximal 
if and only if $\dim X=b_1(X_0)$. 


\subsection{Periodic realizations of crystal lattices}

The notion of periodic realizations of crystal lattices plays a crucial role in 
investigating the natural configurations of crystals into a Euclidean space (see e.g., \cite{KS00}).
Let $X=(V, E)$ be a $d$-dimensional crystal lattice. 
The graph $X$ is identified with a 1-dimensional cell complex in the following manner. 
We take the disjoint union $V \cup (E \times [0, 1])$ and introduce 
the equivalence relation $\sim$ defined by
$o(e) \sim (e, 0)$, $t(e)\sim(e, 1)$ and $(e, t) \sim (\ol{e}, 1-t)$ for $0 \leq t \leq 1$.
Then $X$ is regarded as a cell complex $\big(V \cup (E \times [0, 1]\big)/\sim$
and vertices and unoriented edges are identified with 0-cells and 1-cells, respectively.  
A map $\Phi : X \to \Gamma \otimes \mathbb{R}\cong \mathbb{R}^d$ 
is said to be {\it piecewise linear} if the restriction
$(\Phi|_e) : [0, 1] \to \mathbb{R}^d$ is linear for $e \in E$ and
$(\Phi|_{\ol{e}})(t)=(\Phi|_e)(1-t)$ for $0 \leq t \leq 1$.

\begin{df}[periodic realization]
Let $X=(V, E)$ be a $d$-dimensional crystal lattice.  
A piecewise linear map 
$\Phi : X \to \Gamma \otimes \mathbb{R} \cong \mathbb{R}^d$ is called a 
{\it periodic realization} of $X$  if $\Phi$ satisfies
$$
\Phi(\gamma x)=\Phi(x) + \gamma, \qquad \gamma \in \Gamma, \,\, x \in V,
$$
where $\gamma \in \Gamma$ is identified with $\gamma \otimes 1 \in \Gamma \otimes \mathbb{R}$. 
\end{df} 

Let $\Phi : X \to \Gamma \otimes \mathbb{R}$ a periodic realization of $X$. 
We put 
\begin{equation}\label{dphi}
d\Phi(e):=\Phi\big(t(e)\big) - \Phi\big(o(e)\big), \qquad e \in E.
\end{equation}
Since $d\Phi : E \to \Gamma \otimes \mathbb{R}$
satisfies that 
$d\Phi(\gamma x)=d\Phi(x)$
for $\gamma \in \Gamma$ and $x \in V$, 
it gives rise to an $\mathbb{R}^d$-valued map on $E_0$
with $d\Phi(\ol{e})=-d\Phi(e)$ for $e \in E_0$. 
If we choose a base point $x_* \in V$ and  $\Phi(x_*)$ is fixed, 
then the image $\Phi(X)$ is completely determined by using
the set of vectors $\{d\Phi(e)\}_{e \in E_0}$. 
In this sense, we call  $\{d\Phi(e)\}_{e \in E_0}$ a {\it building block} 
of $\Phi$. 
Conversely, for a given set of vectors 
$\{\va(e)\}_{e \in E_0}$ with $\va(\ol{e})=-\va(e)$ for $e \in E_0$,
we can find a periodic realization $\Phi$ such that 
$\{\va(e)\}_{e \in E_0}$ forms a building block of $\Phi$. 

A periodic realization $\Phi$ of a $d$-dimensional crystal lattice $X$
is said to be {\it non-degenerate} if $\Phi : V \to \mathbb{R}^d$ is injective, 
$d\Phi(e) \neq \bm{0}$ for $e \in E$ and the map
$$
E_x \ni e \longmapsto \frac{d\Phi(e)}{|d\Phi(e)|} \in \mathbb{S}^{d-1}
:=\{\vx \in \mathbb{R}^d \, : \, |\vx|=1\}
$$
is injective for all $x \in V$. This means that edges having the same origin 
never overlap. Otherwise, it is said to be {\it degenerate}. 
Several comprehensible examples of crystal lattices with their non-degenerate 
periodic realizations for $d=1, 2$ are discussed in Section \ref{Sec:Example}.

\section{{\bf Random walks on crystal lattices of finite supports}}

We fix  $d \in \mathbb{N}$. 
Let $X=(V, E)$ be a $d$-dimensional crystal lattice of a finite graph $X_0=(V_0, E_0)$
with an abelian covering transformation group $\Gamma=\mathbb{Z}^d$. 
We write $\mathcal{S}=\{ \gamma_1, \gamma_2, \dots,  \gamma_d\}$ for a set of generators of $\Gamma$. 
We take a non-degenerate periodic realization $\Phi : X \to \Gamma \otimes \mathbb{R} \cong \mathbb{R}^d$
and fix a base point $x_* \in V$ satisfying $\Phi(x_*)=\bm{0}$. 
We also call $x_*$ the {\it origin} of $X$.

\subsection{Multidimensional Shintani zeta functions on crystal lattices}

The aim of this subsection is to define an analogue of 
multidimensional Shintani zeta functions on crystal lattices
and to investigate relations with probability theory. 

As to state the definition, we need a few notations. 
For a vertex $x \in V_0$, we write $\widetilde{x} \in V$ for a lift of $x$ to $V$, that is, an element of the fiber 
$\pi_V^{-1}(x) \subset V$. 
Then we define a set of vectors $\mathcal{J}(\Phi; x)$ by
\begin{equation}\label{jump set Shintani}
\mathcal{J}(\Phi; x):=\Big\{ k\va \,\Big| \, k \in \mathbb{R}, \, \va= \sum_{k=1}^n d\Phi(e_k), \,
c=(e_1, e_2, \dots, e_n) \in \Omega_{\widetilde{x}, n}(X), \, n \in \mathbb{N}\Big\}, 
\end{equation}
where $d\Phi : E \to \mathbb{R}^d$ is defined by \eqref{dphi}. 
Note that $\mathcal{J}(\Phi; x)$ does not depend on the choice of a lift of $x \in V_0$
Then we define the multidimensional Shintani zeta functions on $X$ having a dependence 
of a choice of a vertex in $V_0$, which is a major difference from those defined in \cite{AN2}. 

\begin{df}[multidimensional Shintani zeta function on $X$]
Let $m, r \in \mathbb{N}$, $x \in V_0$, $\vs \in \mathbb{C}^d$ and 
$(n_1, n_2, \dots, n_r) \in \mathbb{Z}_{\ge 0}^r$. 
For $\lambda_{\ell j} \ge 0$
with $\lambda_{\ell j_0} \neq 0$ for some $j_0$, 
$u_j>0$, $\vec{c}_\ell \in \mathcal{J}(\Phi, x)$, where 
$j=1, 2, \dots, r$ and $\ell=1, 2, \dots, m$, 
and a $\mathbb{C}$-valued function
$\theta(n_1, n_2, \dots, n_r)$ satisfying
$|\theta(n_1, n_2, \dots, n_r)|=O\big((n_1+n_2+\cdots+n_r)^\ve\big)$, for any $\ve>0$,  
we define a multidimensional Shintani zeta function $Z_S^{X, \Phi}(x, \vs)$ associated with $x \in V_0$ by
\begin{equation}\label{New Shintani zeta}
Z_S^{X, \Phi}(x, \vs):=\sum_{n_1, n_2, \dots, n_r=0}^\infty \frac{\theta(n_1, n_2, \dots, n_r)}
{\prod_{\ell=1}^{m}\big(\sum_{j=1}^r \lambda_{\ell j}(n_j+u_j)\big)^{\la \vec{c}_\ell, \vs \ra}}, 
\end{equation}
\end{df}

Note that our assumption for $\{\lambda_{\ell j}\}$ 
in the definition above is strictly weaker than that of Definition \ref{Def:Shintani zeta}, which makes it possible to  
represent various kinds of random variables we give afterwards. 
We write $\mathcal{Z}_S^{X, \Phi}(x)$ for the class consisting of all $d$-dimensional
Shintani zeta functions of the form \eqref{New Shintani zeta} associated with $x \in V_0$.  
We put $\vs=\vsg+\ii\vt, \, \vsg, \vt \in \mathbb{R}^d$. 
Under a suitable condition on $\{\lambda_{\ell j}\}$ 
in \eqref{New Shintani zeta}, 
We can show the absolute convergence of 
$Z_S^{X, \Phi}(x, \vs)$ as in the following. 

\begin{pr}
Suppose that $\lambda_{\ell j}>0$ for $\ell=1, 2, \dots, m$ and 
$j=1, 2, \dots, r$ as in \eqref{New Shintani zeta}.
For a fixed $x \in V_0$, the series defined by \eqref{New Shintani zeta} converges absolutely in the region 
$\min\{\mathrm{Re}\,\la\vec{c}_\ell, \vs\ra|\,\ell=1, 2, \dots, m\}>r/m.$
\end{pr}

\begin{proof}
We put 
$$
w=w(x):=\min_{\substack{\ell=1, 2, \dots, m}}\la\vec{c}_\ell, \vsg\ra>r, \quad 
\lambda:=\min_{\substack{\ell=1, 2, \dots, m 
\\ j=1, 2, \dots, r}}\lambda_{\ell j}>0, \quad 
u:=\min_{j=1, 2, \dots, r}u_j>0. 
$$
Then we have
$$
\Big(\sum_{j=1}^r \lambda_{\ell j}
(n_j+u_j)\Big)^{-1}
\leq \lambda^{-1}\Big(\sum_{j=1}^r n_j+ru\Big)^{-1}, \qquad \ell=1, 2, \dots, m.
$$
Hence, for any $0<\ve<w-r$, there exists a 
sufficiently large $C_\ve>0$ such that
\begin{align*}
    &\sum_{n_1, n_2, \dots, n_r=0}^\infty \Bigg|\frac{\theta(n_1, n_2, \dots, n_r)}
    {\prod_{\ell=1}^{m}\big(\sum_{j=1}^r \lambda_{\ell j}(n_j+u_j)\big)^{\la \vec{c}_\ell, \vs \ra}}\Bigg|\\
    &\le \sum_{n_1, n_2, \ldots, n_r =0}^{\infty}
    \frac{C_{\ve}\big(\sum_{j=1}^r n_j+ru)^{\ve}\lambda^{-w}}{\prod_{\ell=1}^m
    \big(\sum_{j=1}^r n_j+ru\big)
    ^{\langle \vc_\ell,\vsg \rangle}}\\
    &=C_{\ve}\sum_{n_1, n_2, \ldots, n_r =0}^{\infty}
    \frac{\lambda^{-w}}{\big(\sum_{j=1}^r n_j+ru\big)
    ^{-\ve+w}} \\
    &\le C_{\ve}\lambda^{-w}\left((ru)
    ^{\ve-w} 
    +\int_0^{\infty}\cdots\int_0^{\infty} \frac{dx_1 \cdots dx_r}
    {(x_1+\cdots+x_r+ru)^{-\ve+w}}\right)\\
    &\le C_{\ve}\lambda^{-w}
    \left((ru)^{\ve-w} 
    +C(ru)^{\ve -w +r}\right)<\infty,
\end{align*}
where
\begin{equation*}
C:=\big\{(w-\ve-1)(w-\ve-2)\cdots(w-\ve-r)\big\}^{-1}>0.
\end{equation*}
Thus, $Z_S^{X, \Phi}(x, \vs)$ converges absolutely in the region $\min\{\mathrm{Re}\,\la\vec{c}_\ell, \vs\ra|\,\ell=1, 2, \dots, m\}>r/m.$
This completes the proof. 
\end{proof}


\subsection{Multidimensional Shintani zeta distributions on crystal lattices}

We define a multidimensional 
Shintani zeta random variable taking values in $X$.  
In the following, let $\theta(n_1, n_2, \dots, n_r), \, 
(n_1, n_2, \dots, n_r) \in  \mathbb{Z}^r_{\ge 0}$,
be a nonnegative or nonpositive definite function. 
We write $\vec{c}_\ell=(c_{\ell 1}, c_{\ell 2}, \dots, c_{\ell d}) \in \mathbb{R}^d$ 
for $\ell=1, 2, \dots, m$.

\begin{df}[multidimensional Shintani zeta distribution on $X$]
We fix $x \in V_0$ and 
$\vsg=\vsg(x) \in \mathbb{R}^d \setminus \{\bm{0}\}$. 
An $\mathbb{R}^d$-valued random variable 
$\mathcal{Y}_{\vsg}=\mathcal{Y}_{\vsg(x)}$ is called a multidimensional Shintani zeta random variable associated with $x \in V_0$ if  
\begin{align}
&{\bf{P}}\Bigg(\mathcal{Y}_{\vsg}=
\Bigg(-\sum_{\ell=1}^m c_{\ell 1}\log\Big(\sum_{j=1}^r \lambda_{\ell j}(n_j+u_j)\Big), \dots, 
-\sum_{\ell=1}^m c_{\ell d}\log\Big(\sum_{j=1}^r \lambda_{\ell j}(n_j+u_j)\Big)\Bigg)\Bigg)\nn\\
&=\frac{\theta(n_1, n_2, \dots, n_r)}{Z_S^{X, \Phi}(x, \vsg)}\prod_{\ell=1}^m
\Big(\sum_{j=1}^r \lambda_{\ell j}(n_j+u_j)\Big)^{-\la \vec{c}_\ell, \vsg\ra} \label{Shintani zeta dist}
\end{align}
for $(n_1, n_2, \dots, n_r) \in \mathbb{Z}_{\ge 0}^r$. 
In particular, a $V$-valued random variable $\mathcal{X}_{\vsg}$ 
is called a multidimensional Shintani zeta random variable on $X$ associated with $x \in V_0$ when $\mathcal{Y}_{\vsg}:=\Phi(\mathcal{X}_{\vsg})$ is given by \eqref{Shintani zeta dist}
and $\mathcal{Y}_{\vsg} \in \Phi(V)$. 
\end{df}
We can easily verify that the above distribution is 
actually a probability distribution on $\mathbb{R}^d$ since 
the right-hand side of \eqref{Shintani zeta dist} is non-negative and 
the sum of the right-hand side of \eqref{Shintani zeta dist} over 
all $(n_1, n_2, \dots, n_r) \in \mathbb{Z}_{\ge 0}^r$ is equal to one. We note that, if $\lambda_{\ell j}>0$ for $\ell=1, 2, \dots, m$ and $j=1, 2, \dots, r$, then we can choose $\vsg(x)$ as an element
in the region $\min\{\la\vec{c}_\ell, \vsg\ra|\,\ell=1, 2, \dots, m\}>r/m.$

Then the characteristic function of each $\mathcal{Y}_{\vsg}$ is obtained as the desired form. 

\begin{tm}
Let $\mathcal{Y}_{\vsg}=\mathcal{Y}_{\vsg(x)}$ be a 
multidimensional Shintani zeta random variable associated with $x \in V_0$. 
Then the characteristic function $f_{\vsg}^{X, \Phi}(x, \cdot)$ 
of $\mathcal{Y}_{\vsg}$ is given by
$$
f_{\vsg}^{X, \Phi}(x, \vt)=\frac{Z_S^{X, \Phi}(x, \vsg+\ii \vt)}{Z_S^{X, \Phi}(x, \vsg)}, \qquad
\vt \in \mathbb{R}^d.  
$$
\end{tm}

\begin{proof} 
The proof is straightforward. For any $\vt \in \mathbb{R}^d$, we obtain
$$
\begin{aligned}
f_{\vsg}^{X, \Phi}(x, \vt)&=\sum_{n_1, n_2, \dots, n_r=0}^\infty
\e^{\ii \la \vt, \mathcal{Y}_{\vsg}\ra}
\frac{\theta(n_1, n_2, \dots, n_r)}{Z_S^{X, \Phi}(x, \vsg)}\prod_{\ell=1}^m
\Big(\sum_{j=1}^r \lambda_{\ell j}(n_j+u_j)\Big)^{-\la \vec{c}_\ell, \vsg\ra}\\
&=\frac{Z_S^{X, \Phi}(x, \vsg+\ii \vt)}{Z_S^{X, \Phi}(x, \vsg)}.
\end{aligned}
$$
This completes the proof. 
\end{proof}

\subsection{Random walks generated by multidimensional Shintani zeta functions }

We define random walks on crystal lattices whose range is finite in this subsection.
In particular, such random walks have been well-studied and some limit theorems 
for them such as central limit theorems and large deviation principles have established.
See e.g., \cite{KS06} and \cite{IKK}. 
However, we should note that only nearest-neighbor random walks are discussed among the papers. 

In this subsection, we define a class of finite range 
random walks on crystal lattices generated by 
multidimensional Shintani zeta functions,
which includes various kinds of random walks 
regardless of whether they admit nearest-neighbor jumps or not. 
The following theorem tells us that the usual random variables which represent the 
each step of random walks can be written by a multidimensional Shintani
zeta ones. 
\begin{tm}
Let $\va_1, \va_2, \dots, \va_m \in \mathcal{J}(\Phi; x)$ for some $x \in V_0$. 
We define a $V$-valued random variable $\mathcal{X}$ by
$$
{\bf{P}}\big(\Phi(\mathcal{X})=\va_\ell\big)=\beta_\ell, \qquad \ell=1, 2, \dots, m,
$$
where $\beta_\ell, \, \ell=1, 2, \dots, m$, are nonnegative real number satisfying 
$\beta_1+\beta_2+\cdots+\beta_m=1$. 
Then $\mathcal{X}$ is a multidimensional Shintani zeta random variable on $X$
associated with $x \in V_0$.  
\end{tm}

\begin{proof}
It is easily seen that the characteristic function of $\Phi(\mathcal{X})$ is given by
$$
{\bf{E}}[e^{\ii\la \vt, \Phi(\mathcal{X})\ra}]=\sum_{\ell=1}^m 
\beta_\ell \e^{\ii \la \va_\ell, \vt\ra}, \qquad \vt \in \mathbb{R}^d. 
$$
We now consider the following multidimensional Shintani zeta function.  
Let $r=m$ and $j_1, j_2, \dots, j_m \in \mathbb{N}$ be $m$ distinct integers bigger than 1. 
We put $\vec{c}_\ell=-(\log j_\ell)^{-1}\va_\ell$, 
$\lambda_{\ell j}=\delta_{\ell j}$ and $u_j=1$ for 
$\ell, j=1, 2, \dots, m$, where $\delta_{\ell j}$ is the usual Kronecker's delta.  
We take an arbitrary vector 
$\vec{\sigma} \in \mathbb{R}^d \setminus \{\bm{0}\}$ 
and put
$$
\begin{aligned}
&\theta(n_1, n_2, \dots, n_m)\\
&:=\begin{cases}
\beta_{\ell}\e^{- \la \va_\ell, \vsg\ra} & 
\text{if }(n_1, \dots, n_\ell, \dots, n_m)=(0, \dots, j_\ell-1, \dots, 0), 
\, \ell=1, 2, \dots, m\\
0 & \text{otherwise}
\end{cases},
\end{aligned}
$$
which gives a non-negative  definite function satisfying
$|\theta(n_1, n_2, \dots, n_m)|=O\big((n_1+n_2+\cdots+n_m)^\ve\big)$ 
for any $\ve>0$.
Then we have
$$
\begin{aligned}
Z_S^{X, \Phi}(x, \vs) 
&= \sum_{n_1, n_2, \dots, n_m=0}^\infty \frac{\theta(n_1, n_2, \dots, n_m)}
{\prod_{\ell=1}^{m}\big( \sum_{j=1}^m \lambda_{\ell j}(n_j+u_j)\big)
^{\la \vec{c}_\ell, \vs \ra}}\\
&=\sum_{\ell=1}^m \frac{\beta_{\ell}\e^{- \la \va_\ell, \vsg\ra}}
{j_\ell^{\la \vc_\ell, \vs\ra}}
 =\sum_{\ell=1}^m \frac{\beta_{\ell}\e^{- \la \va_\ell, \vsg\ra}}
 {\e^{-\la \va_\ell, \vs\ra}}
 =\sum_{\ell=1}^m \beta_\ell 
 \e^{\la \va_\ell, \vs-\vsg\ra}, \qquad \vs \in \mathbb{C}^d.  
\end{aligned}
$$
Thus, the assumption $\sum_{\ell=1}^m \beta_\ell=1$ implies that
$$
f_{\vsg}^{X, \Phi}(\vt)=\frac{Z_S^{X, \Phi}(x, \vsg+\ii \vt)}{Z_S^{X, \Phi}(x, \vsg)}
=\frac{\sum_{\ell=1}^m \beta_\ell 
 \e^{\ii\la \va_\ell, \vt\ra}}
 {\sum_{\ell=1}^m \beta_\ell}
 =\sum_{\ell=1}^m 
\beta_\ell \e^{\ii \la \va_\ell, \vt\ra}={\bf{E}}[\e^{\ii\la \vt, \Phi(\mathcal{X})\ra}]
$$
for $\vt \in \mathbb{R}^d$, which completes the proof. 
\end{proof}

Let $x_* \in V$ be an origin of $X$.
Then, we define a finite range random walk on $X$ starting from $x_*$
which is generated by multidimensional Shintani zeta functions as in the following. 

\begin{df}[finite range random walk on $X$ generated by multidimensional Shintani zeta functions on $X$]\label{Def:RW-Shintani}
For each $x \in V_0$, we choose arbitrary $\vsg(x) \in \mathbb{R}^d\setminus \{\bm{0}\}$ and $m(x) \in \mathbb{N}$. Let $\va_1(x), \va_2(x), \dots, \va_{m(x)}(x) \in \mathcal{J}(\Phi; x)$ be arbitrary $m(x)$ vectors.
A sequence of $V$-valued independent random variables $\{\mathcal{W}_n\}_{n=0}^\infty$ is called a finite range random walk 
generated by multidimensional Shintani zeta functions on $X$ if 
$\mathcal{W}_0=x_*$ a.s.\,and each 
$\mathcal{W}_n, \, n \in \mathbb{N}$, is a
multidimensional Shintani zeta random variable on $X$ associated with
$\pi(\mathcal{W}_{n-1}) \in V_0$ such that
\begin{align}\label{RW}
{{\bf E}}[\e^{\ii \la \vt, \Phi(\mathcal{W}_n)\ra}]
=
f_{\vsg(\pi(x_0))}^{X, \Phi}(\vt)
f_{\vsg(\pi(\mathcal{W}_1))}^{X, \Phi}(\vt) \cdots
f_{\vsg(\pi(\mathcal{W}_{n-1}))}^{X, \Phi}(\vt), 
\qquad \vt \in \mathbb{R}^d. 
\end{align}
\end{df}

Note that our random walk is defined by its characteristic function as in \eqref{RW}.
All the increments $\mathcal{W}_1 -\mathcal{W}_0, \dots,
\mathcal{W}_n -\mathcal{W}_{n-1}$ at each steps are independent but the 
corresponding distribution of $\mathcal{W}_i -\mathcal{W}_{i-1}$ depend on which $x=\pi(\mathcal{W}_{i-1}) \in V_0$, $i=1, 2, \dots ,n$, they are at.
Therefore, we define our random walks by characteristic functions as to make things simple.

\section{{\bf Random walks on crystal lattices of infinite supports}}

\subsection{Multidimensional Euler products on crystal lattices }

The basic settings are same as in the previous section. Throughout this section, 
we assume the following. 
$$
|V_0|=1, \text{ {\it that is, $X$ is a Cayley graph of $\Gamma=\mathbb{Z}^d$ with a generating set $\mathcal{S}$}.}  
$$ 
Since the $d$-dimensional crystal lattice $X$ can be regarded as a subset of $d$-dimensional
Euclidean space $\mathbb{R}^d$ through a periodic realization $\Phi$, 
we expect that $d$-dimensional  Euler products on $X$
may be defined
in the same way as the multidimensional polynomial Euler products (\ref{ZE}) 
on $\mathbb{R}^d$. 
However, since the class $\mathcal{Z}_E^{(d)}$  is too large to be treatable in graph settings, 
we introduce its suitable subclass $\mathcal{Z}_{f\!E}^{X, \Phi}$ consisting 
of certain finite Euler products. 
Then the compound Poisson zeta distributions on crystal lattices generated by the finite 
Euler products can be defined.  
This will give concrete and direct ways to treat random walks of infinite ranges with values in 
infinite periodic graphs.

Let $m \in \mathbb{N}$.  
We agree that the simplest way to define multidimensional polynomial Euler products on 
a crystal lattice $X$ would be 
\begin{equation}\label{ZEG}
Z_{E}^{X, \Phi}(\vs)=\prod_{p \in \mathbb{P}}\prod_{\ell=1}^{m} 
\big( 1 - \alpha_{\ell}(p)
p^{-\la \va_{\ell}, \vs \ra}\big)^{-1}, \qquad \vs \in \mathbb{C}^d,
\end{equation}
where 
$-1 \le \alpha_{\ell}(p) \leq 1$ for 
$\ell=1, 2, \dots, m$ and $p \in \mathbb{P}$, and 
$\va_{\ell} \in \Phi(V) \subset \mathbb{R}^d$ for $\ell=1, 2, \dots, m$.
We now consider a function given by 
$$
f_{\vsg}^{X, \Phi}(\vt):=
\frac{Z_E^{X, \Phi}(\vsg+\ii \vt)}{Z_E^{X, \Phi}(\vsg)}, \qquad \vt \in \mathbb{R}^d,
$$
where $\vsg$ satisfies $\min\{\la \va_\ell, \vec{\sigma}\ra|\, \ell=1, 2, \dots, m\}>1$. 
As in Theorem \ref{AN-zeta}, there is no doubt that
the function $f_{\vec{\sigma}}^{X, \Phi}$ is also to be a compound Poisson 
characteristic function on $\mathbb{R}^d$ under some suitable situations. 
Moreover, since the periodic realization $\Phi : X \to \mathbb{R}^d$ is non-degenerate, 
we also expect that the characteristic function $f_{\vec{\sigma}}^{X, \Phi}$
can be regarded as a function on the crystal lattice $X$ and 
that $f_{\vec{\sigma}}^{X, \Phi}$ induces a ``compound Poisson random variable''
with values in $X$ having infinitely many mass points. 
Whereas, such ideas do not work well in that the pull-backs of delta masses may not lie on vertices of the crystal lattice $X$ in general.

Therefore, we need to find a subclass of $\mathcal{Z}_E^{(d)}$ in which such ideas
does work properly. 
We define a set of vectors $\mathcal{J}$  by
$$
\mathcal{J}:=\big\{\va= \ve_1 \gamma_{i_1} +\ve_2 \gamma_{i_2}
+ \cdots +\ve_k \gamma_{i_k}\, \big| \, \gamma_{i_j} \in \mathcal{S}, 
\, \ve_j=\pm 1, \, j =1, 2, \dots, k, \, k \in \mathbb{N}\big\},
$$
where we identify each $\gamma_{i_\ell} \in \mathcal{S}$ 
with $\gamma_{i_\ell}=\gamma_{i_\ell} \otimes 1 \in \mathbb{R}^d$. 
We should emphasize that every $\Phi(x), \, x \in V,$ can be represented 
as an element of $\mathcal{J}$ thanks to $|V_0|=1$.  
Then we define the following.

\begin{df}[multidimensional finite Euler product on $X$]
\label{Def-finite Euler}
Let $m \in \mathbb{N}$ and $\vs \in \mathbb{C}^d$. 
For $-1 \le \alpha_\ell(p) \le 1$ and $\vec{a}_{\ell} \in \mathcal{J}$,  
$\ell=1, 2, \dots, m$,  
we define a multidimensional Euler product on $X$ by 
\begin{equation}\label{finite-Euler-1}
Z_{f\!E}^{X, \Phi}(\vs)=\prod_{\ell=1}^{m} 
\big( 1 - \alpha_{\ell}
\e^{-\la \va_{\ell}, \vs \ra}\big)^{-1}.
\end{equation}
\end{df} 

We refer to \cite{AN1} for a related study of 
multivariate finite Euler products on $\mathbb{R}^2$
in probabilistic view. 
We denote by $\mathcal{Z}_{f\!E}^{X, \Phi}$ the set of all functions of the form \eqref{finite-Euler-1}. 
Then the following relation with $\mathcal{Z}_E^{(d)}$ is verified. 

\begin{pr}
We have 
$\mathcal{Z}_{f\!E}^{X, \Phi} \subset \mathcal{Z}_E^{(d)}$. 
\end{pr}

\begin{proof}
We take a function $Z_{f\!E}^{X, \Phi}(\vec{s}) \in \mathcal{Z}_{f\!E}^{X, \Phi}$ 
given as in the following.
Let us fix distinct $m$ prime numbers $p_1, p_2, \dots, p_m \in \mathbb{P}$. 
We define $\alpha_\ell(p)$ for 
$\ell=1, 2, \dots, m$ and $p \in \mathbb{P}$ by
\begin{equation}\label{p-reduction}
\alpha_\ell(p):=\begin{cases}
\alpha_\ell & \text{if } p=p_\ell \\
0 & \text{otherwise}
\end{cases}.
\end{equation}
Then it is clear that $-1 \le \alpha_\ell(p) \leq 1$ for $\ell=1, 2, \dots, m$ and $p \in \mathbb{P}$. 
Moreover, we put $\vec{b}_\ell:=(\log p_\ell)^{-1}\vec{a}_{\ell}$ for $\ell=1, 2, \dots, m$. 
Then we obtain
$$
Z_{f\!E}^{X, \Phi}(\vec{s})=\prod_{\ell=1}^{m} 
\big( 1 - \alpha_{\ell}
p_\ell^{-(\log p_\ell)\la \va_{\ell}, \vs \ra}\big)^{-1}
=\prod_{p \in \mathbb{P}}\prod_{\ell=1}^{m} 
\big( 1 - \alpha_{\ell}(p)
p^{-\la \vec{b}_{\ell}, \vs \ra}\big)^{-1} \in \mathcal{Z}_E^{(d)},
$$
which completes the proof. 
\end{proof}

\subsection{Conditions to generate characteristic functions}

We provide a necessary and sufficient condition for some 
multidimensional finite Euler products to generate 
compound Poisson characteristic functions on $\mathbb{R}^d$
by following the claim which was  discussed in \cite{AN3}. 
Our aim in this section is to prove the following. 

\begin{tm} \label{New-zeta-1}
Let $Z_{f\!E}^{X, \Phi}(\vec{s}) \in \mathcal{Z}_{f\!E}^{X, \Phi}$. Suppose that
$\va_{\ell}, \, \ell=1, 2, \dots, m$, 
 in {\rm (\ref{finite-Euler-1})} satisfy {\bf (LI)}. 
We also suppose that $\vsg \in \mathbb{R}^d$ satisfies 
$\min\{\la \va_{\ell}, \vsg \ra| \, \ell=1, 2, \dots, m\}>0$. 
Then, the function
$$
f_{\vsg}^{X, \Phi}(\vt)=\frac{Z_{f\!E}^{X, \Phi}(\vsg+\ii \vt)}{Z_{f\!E}^{X, \Phi}(\vsg)}, \qquad \vt \in \mathbb{R}^d,
$$
is a characteristic function on $\mathbb{R}^d$ if and only if 
$\alpha_{\ell} \geq 0$
for all $\ell=1, 2, \dots, m$. 
Moreover, $f_{\vsg}^{X, \Phi}(\vt)$ is a compound Poisson characteristic function with
its finite L\'evy measure $N_{\vsg}^{X, \Phi}(d\vx)$ on $\mathbb{R}^d$ given by
\begin{equation}\label{New-zeta-1-Levy}
N_{\vsg}^{X, \Phi}(d\vx)=
\sum_{r=1}^\infty \sum_{\ell=1}^m
 \frac{1}{r}\alpha_{\ell}^r \e^{-r\la \va_{\ell}, \vsg\ra}
 \delta_{r \va_{\ell}}(d\vx).
\end{equation}
\end{tm}

The following lemma plays an
essential role
in the proof of Theorem \ref{New-zeta-1}. 
We give the proof by applying the (first form of) Kronecker
approximation theorem. 

\begin{lm}\label{Key Lemma}
If there exist $\ell, \, \ell=1, 2, \dots, m,$ satisfying
$\alpha_\ell<0$, then 
there exists $\vt_0 \in \mathbb{R}^d$ such that 
$|f_{\vsg}^{X, \Phi}(\vt_0)|>1$. 
\end{lm}

\begin{proof}
Let us put 
$$
A^+:=\{\ell \in \{1, 2, \dots, m\} \, | \, \alpha_\ell \ge 0\}, \qquad A^-:=\{\ell \in \{1, 2, \dots, m\} \, | \, \alpha_\ell<0\}.
$$
Let $\omega_1, \omega_2, \dots, \omega_m$ with $\omega_1=1$ be algebraic real
numbers which are linearly independent over the rationals. Since $\va_1, \va_2, \dots, \va_m$ satisfy {\bf (LI)}, 
there exists $\vt_0 \in \mathbb{R}^d$ such that 
$\la \va_\ell, \vt_0 \ra=\omega_\ell$
for $\ell=1, 2, \dots, m$.
We define a function
$D : \mathbb{R} \to \mathbb{R}$ by 
$D(T):=\log |f_{\vsg}^{X, \Phi}(T\vt_0)|, \, T \in \mathbb{R}$. 
A direct computation gives us
\begin{align}
D(T)&=\frac{1}{2}\log\Bigg|\frac{Z_{f\!E}^{X, \Phi}(\vsg-\ii T\vt_0)}{Z_{f\!E}^{X, \Phi}(\vsg)} \cdot \frac{Z_{f\!E}^{X, \Phi}(\vsg+\ii T\vt_0)}{Z_{f\!E}^{X, \Phi}(\vsg)}\Bigg|\nn\\
&=\frac{1}{2}\sum_{r=1}^\infty \sum_{\ell=1}^m 
\frac{1}{r}\alpha_\ell^r \e^{-r\la \va_\ell, \vsg\ra}
\Big(\e^{\ii r T \la \va_\ell, \vt_0 \ra}+
\e^{-\ii rT \la \va_\ell, \vt_0 \ra}-2\Big)\nn\\
&=\frac{1}{2}\sum_{r=1}^\infty \sum_{\ell=1}^m 
\frac{1}{r}\alpha_\ell^r \e^{-r\la \va_\ell, \vsg\ra}
\Big(\e^{\ii rT \omega_\ell}+
\e^{-\ii rT \omega_\ell}-2\Big), \qquad T \in \mathbb{R}.
\label{Proof:Lemma4.4-1}
\end{align}
Let $K \in \mathbb{N}$. We
decompose the right-hand side of \eqref{Proof:Lemma4.4-1}
as 
$$
\begin{aligned}
D(T)&=\frac{1}{2}\sum_{r=2K+1}^\infty \sum_{\ell =1}^m 
\frac{1}{r}\alpha_\ell^r \e^{-r\la \va_\ell, \vsg\ra}
\Big(\e^{\ii rT \omega_\ell}+
\e^{-\ii rT \omega_\ell}-2\Big)\\
&\hspace{1cm}+\frac{1}{2}
\sum_{r=1}^{2K} \sum_{\ell \in A^+}\frac{1}{r}\alpha_\ell^r \e^{-r\la \va_\ell, \vsg\ra}
\Big(\e^{\ii rT \omega_\ell}+
\e^{-\ii rT \omega_\ell}-2\Big)\\
&\hspace{1cm}+\frac{1}{2}
\sum_{k=1}^K\sum_{\ell \in A^-}\frac{1}{2k}\alpha_\ell^{2k} \e^{-2k\la \va_\ell, \vsg\ra}
\Big(\e^{\ii 2kT \omega_\ell}+
\e^{-\ii 2kT \omega_\ell}-2\Big)\\
&\hspace{1cm}+\frac{1}{2}
\sum_{k=1}^K\sum_{\ell \in A^-}\frac{1}{2k-1}\alpha_\ell^{2k-1} \e^{-(2k-1)\la \va_\ell, \vsg\ra}
\Big(\e^{\ii (2k-1)T \omega_\ell}+
\e^{-\ii (2k-1)T \omega_\ell}-2\Big)\\
&=:I_1(T)+I_2(T)+I_3(T)+I_4(T). 
\end{aligned}
$$
We easily see that, for any $\ve_1>0$, there exists a 
sufficiently large $K \in \mathbb{N}$ such that $|I_1(T)|<\ve_1$ for all $T \in \mathbb{R}$. 
Since $\omega_1, \omega_2, \dots, \omega_m$ are linearly independent over the rationals, 
Kronecker's approximation theorem (cf.\,\cite[Theorem 7.9]{Apostol}) implies that, for any $\ve_2>0$ independent of $\ve_1$, there exists $T_0 \in \mathbb{R}$
such that 
$$
|\e^{\ii T_0\omega_\ell}-1|<\ve_2, \quad \ell \in A^+, \qquad 
|\e^{\ii T_0\omega_\ell}+1|<\ve_2, \quad \ell \in A^-.
$$
Then we can give estimates of $I_2(T_0), I_3(T_0)$ and 
$I_4(T_0)$ as in the following. Since
$$
|\e^{\ii rT_0 \omega_\ell}-1|
=|\e^{\ii T_0 \omega_\ell}-1||\e^{\ii (r-1)T_0 \omega_\ell}+
\e^{\ii (r-2)T_0 \omega_\ell}+\cdots+1|\le r\ve_2 \le 2K\ve_2, 
$$
for $r=1, 2, \dots, 2K$, we have
$$
\begin{aligned}
I_2(T_0) &\ge -2K\ve_2
\sum_{r=1}^{2K} \sum_{\ell \in A^+}\frac{1}{r}\alpha_\ell^r \e^{-r\la \va_\ell, \vsg\ra}.
\end{aligned}
$$
Similarly, since it holds that
$$
|\e^{\ii 2kT_0 \omega_\ell}-1|
=|\e^{\ii T_0 \omega_\ell}-1||\e^{\ii T_0 \omega_\ell}+1||\e^{\ii (2k-2)T_0 \omega_\ell}+
\e^{\ii (2k-4)T_0 \omega_\ell}+\cdots+1|\le 2k\ve_2 \le 2K\ve_2,
$$
for $k=1, 2, \dots, K$, we also have
$$
\begin{aligned}
I_3(T_0) &\ge -2K\ve_2\sum_{k=1}^{K} \sum_{\ell \in A^-}\frac{1}{2k}\alpha_\ell^{2k} \e^{-2k\la \va_\ell, \vsg\ra}.
\end{aligned}
$$
Moreover, by noting 
$$
|\e^{\ii (2k-1)T_0 \omega_\ell}+1|
=|\e^{\ii T_0 \omega_\ell}+1||\e^{\ii (2k-2)T_0 \omega_\ell}+
\e^{\ii (2k-3)T_0 \omega_\ell}+\cdots+1|\le (2k-1)\ve_2 \le 2K\ve_2,
$$
for $k=1, 2, \dots, K$, we have
$$
\begin{aligned}
I_4(T_0) &\ge -(2K+2)\ve_2\sum_{k=1}^{K} \sum_{\ell \in A^-}\frac{1}{2k-1}\alpha_\ell^{2k-1} \e^{-(2k-1)\la \va_\ell, \vsg\ra}.
\end{aligned}
$$
By putting them all together, we obtain
$$
\begin{aligned}
D(T_0) &> -\ve_1-2K\ve_2
\sum_{r=1}^{2K} \sum_{\ell \in A^+}\frac{1}{r}\alpha_\ell^r \e^{-r\la \va_\ell, \vsg\ra}-2K\ve_2\sum_{k=1}^{K} \sum_{\ell \in A^-}\frac{1}{2k}\alpha_\ell^{2k} \e^{-2k\la \va_\ell, \vsg\ra}\\
&\hspace{1cm}+(2K\ve_2-2)\sum_{k=1}^{K} \sum_{\ell \in A^-}\frac{1}{2k-1}\alpha_\ell^{2k-1} \e^{-(2k-1)\la \va_\ell, \vsg\ra}. 
\end{aligned}
$$
Suppose that $\ve_1$ and $\ve_2$ are sufficiently small so that
$K\ve_2<\ve_1$. Then we obtain $D(T_0)>0$, which 
completes the proof.  
\end{proof}

Next we prove Theorem \ref{New-zeta-1}.

\begin{proof}[Proof of Theorem \ref{New-zeta-1}]

By applying Lemma \ref{Key Lemma}, 
if there exists $\ell, \, \ell=1, 2, \dots, m,$ such that 
$\alpha_\ell<0$, then $f_{\vsg}^{X, \Phi}$ is not a 
characteristic function on $\mathbb{R}^d$, 
since all characteristic functions $f$ have to 
satisfy $|f(\vt)| \le 1$, 
$\vt \in \mathbb{R}^d$.
Therefore, we have only to show that 
$f_{\vsg}^{X, \Phi}$ is a compound Poisson characteristic function with a finite L\'evy measure
$N_{\vsg}^{X, \Phi}$ given by \eqref{New-zeta-1-Levy},
if $\alpha_\ell \ge 0$ for $\ell=1, 2, \dots, m$. 

Since $\alpha_\ell \ge 0$ for $\ell=1, 2, \dots, m$, 
$N_{\vsg}^{X, \Phi}$ is clearly a measure on $\mathbb{R}^d$. 
If $v:=\min\{\la \vec{a}_\ell, \vec{\sigma}\ra \, | \, \ell=1, 2, \dots, m\}>0$, 
then we obtain
$$
\begin{aligned}
\log f_{\vec{\sigma}}^{X, \Phi}(\vt)&=\log \frac{Z_E^{X, \Phi}(\vsg+\ii\vt)}{Z_E^{X, \Phi}(\vsg)}
=\sum_{r=1}^\infty \sum_{\ell=1}^m \frac{1}{r}\alpha_\ell^r \e^{-r\la \va_\ell, \vsg\ra}
(\e^{-\ii r\la \va_\ell, \vt\ra}-1)\\
&=\int_{\mathbb{R}^d} (\e^{-\ii\la \vt, \vx\ra}-1) \, N_{\vsg}^{X, \Phi}(d\vx),  \qquad \vt \in \mathbb{R}^d,
\end{aligned}
$$
and
$$
\begin{aligned}
N_{\vsg}^{X, \Phi}(\mathbb{R}^d) &= \int_{\mathbb{R}^d} \sum_{r=1}^\infty
\sum_{\ell=1}^m \frac{1}{r}\alpha_\ell^r \e^{-r\la \va_\ell, \vsg\ra}\delta_{r\va_\ell}(d\vx)\\
&\leq m\sum_{r=1}^\infty \frac{1}{r}\e^{-rv} 
\leq m \sum_{r=1}^\infty \e^{-rv} =\frac{m}{\e^v-1}<+\infty.
\end{aligned}
$$
Therefore, $f_{\vsg}^{X, \Phi}$ is 
an infinitely divisible characteristic function on $\mathbb{R}^d$ and $N_{\vsg}^{X, \Phi}$ is a finite 
L\'evy measure
on $\mathbb{R}^d$ by Proposition \ref{LK-formula}.
Let $\mathcal{X}$ be an 
$\mathbb{R}^d$-valued random variable by
\begin{equation}\label{FZ-CP-1step-rv}
{\bf{P}}(\mathcal{X}=-r \va_\ell):=\frac{1}{N_{\vec{\sigma}}^{X, \Phi}(\mathbb{R}^d)}\cdot
\frac{\alpha_\ell^r \e^{-r\la \va_\ell, \vsg \ra}}{r}, \qquad r \in \mathbb{N}, \, \ell=1, 2, \dots, m. 
\end{equation}
Suppose that $\mathcal{X}_1, \mathcal{X}_2, \dots$ are independent copies of $\mathcal{X}$ and 
$K_{\vec{\sigma}}^{X, \Phi}$ is a Poisson random variable with a parameter $\lambda:=N_{\vec{\sigma}}^{X, \Phi}(\mathbb{R}^d)>0$
independent of $\mathcal{X}$. Then we easily verify that
$$
f_{\vsg}^{X, \Phi}(\vt)=\exp\Big(\lambda\big(
\mathbf{E}[\e^{\ii\la \vt, \mathcal{X}\ra}]-1\big)\Big), \qquad \vt \in \mathbb{R}^d,
$$
which implies that $f_{\vsg}^{X, \Phi}$ is a 
compound Poisson characteristic
function with a finite L\'evy measure $N_{\vec{\sigma}}^{X, \Phi}$ on $\mathbb{R}^d$. 
This completes the proof.
\end{proof}

What is important in Theorem \ref{New-zeta-1} is that all delta masses in \eqref{New-zeta-1-Levy}
lie on the set $\Phi(V) \subset \mathbb{R}^d$. Therefore, we can pull the finite L\'evy measure
$N_{\vsg}^{X, \Phi}(d\vx)$ on $\mathbb{R}^d$ back to the crystal lattice $X$ and can construct 
a ``compound Poisson random variable'' with values in $X$ 
corresponding to $N_{\vsg}^{X, \Phi}(d\vx)$ properly.

Let $\vsg \in \mathbb{R}^d$ be a vector staisfying 
$\min\{\la \va_{\ell}, \vsg \ra| \, \ell=1, 2, \dots, m\}>0$.
Then we reach the following definition. 

\begin{df}[compound Poisson random variable on $X$ generated by a multidimensional Euler product]
A $V$-valued random variable $\mathcal{Y}_{\vsg}$ is called a
compound Poisson random variable on $X$ 
generated by $Z_{f\!E}^{X, \Phi} \in \mathcal{Z}_{f\!E}^{X, \Phi}$ if it satisfies
\begin{equation}\label{FZ-CP}
\Phi(\mathcal{Y}_{\vsg})=\mathcal{X}_1+\mathcal{X}_2+\cdots+\mathcal{X}_{K_{\vec{\sigma}}^{X, \Phi}} 
\end{equation}
for some sequence of independent and identically distributed 
random variables $\{\mathcal{X}_n\}_{n=1}^\infty$
whose distibution is given by \eqref{FZ-CP-1step-rv} and for some 
Poisson random variable $K_{\vsg}^{X, \Phi}$ independent of $\{\mathcal{X}_n\}_{n=1}^\infty$.
\end{df}

This definition tells us that a compound Poisson random variable with values in a crystal lattice 
can be generated by each element 
in $\mathcal{Z}_{f\!E}^{X, \Phi}$ through a non-degenerate periodic realization $\Phi$. Moreover, 
the L\'evy measure $\nu_{\vec{\sigma}}^X(dx)$ on $V$ 
corresponding to $\mathcal{Y}_{\vsg}$ is heuristically written as
$$
\nu_{\vec{\sigma}}^X(dx)=\sum_{\ell=1}^m \frac{1}{r}\alpha_\ell^r \e^{-r \la \vec{a}_\ell, \vec{\sigma}\ra}
\delta_{\gamma(\ell)^r}(dx).
$$ 
Here, for a vector $\vec{a}_\ell \in \mathcal{J}, \, \ell=1, 2, \dots, m$, there exist 
$\gamma_{{\ell_1}}, \gamma_{{\ell_2}}, \dots, \gamma_{{\ell_{k_\ell}}} \in \mathcal{S}$ such that 
$$
\va_\ell=\ve_{1} \gamma_{\ell_1} + \ve_{2}\gamma_{\ell_2} 
\cdots +\ve_{k_\ell} \gamma_{\ell_{k_\ell}},\quad
\ve_{j}=\pm 1, \, j=1, 2, \dots, k_\ell,
$$
and we put 
$\gamma(\ell)=\gamma_{\ell_1}^{\ve_1}  \gamma_{\ell_2}^{\ve_{2}}
\cdots  \gamma_{\ell_{k_\ell}}^{\ve_{k_\ell}} \in \Gamma$.

The next theorem provides a necessary and sufficient condition for a given random variable 
$\mathcal{X}$ with values in the set $\{r \vec{a}_\ell \, | \, r \in \mathbb{N}, \, \ell=1, 2, \dots, m\}$
to generate a compound Poisson random variable on $X$
under {\bf (LI)}. 

\begin{tm}
Suppose that $m$ vectors $\va_1, \va_2, \dots, \va_m \in \mathcal{J}$ satisfy 
{\bf (LI)}. 
Let $\mathcal{X}_1, \mathcal{X}_2, \dots$ be $\mathbb{R}^d$-valued 
independent and identically distributed 
random variables with
$$
{\bf{P}}(\mathcal{X}_1=r\va_\ell)=\beta(r, \ell), \qquad r \in \mathbb{N}, \, \ell=1, 2, \dots, m,
$$
where $\beta(r, \ell), \, r \in \mathbb{N}, \, \ell=1, 2, \dots, m$, 
are nonnegative real numbers satisfying
\begin{equation}\label{sum=1}
\sum_{r=1}^\infty\sum_{\ell=1}^m \beta(r, \ell)=1. 
\end{equation}
Then a $V$-valued random variable $\mathcal{Y}$ is a
compound Poisson random variable on $X$
generated by $Z_{f\!E}^{X, \Phi} \in \mathcal{Z}_{f\!E}^{X, \Phi}$, that is, each $\beta(r, \ell)$
is given by the right-hand side of \eqref{FZ-CP-1step-rv} and Equation
\eqref{FZ-CP} holds for some $\vec{\sigma} \in \mathbb{R}^d$ with 
$\min\{\la \vec{a}_\ell, \vec{\sigma}\ra \, | \, \ell=1, 2, \dots, m\}>0$ and some
Poisson random variable $K_{\vec{\sigma}}^{X, \Phi}$ 
independent of $\{\mathcal{X}_i\}_{i=1}^\infty$ if and only if the sequence 
$\{r\beta(r, \ell)\}_{r=1}^{\infty}$ is a geometric sequence 
with the common ratio $0<A_\ell<1$ 
for $\ell=1, 2, \dots, m$. 
\end{tm}

\begin{proof}
Suppose that $\mathcal{Y}$ is a
compound Poisson random variable on $X$. Then we have
$$
r\beta(r, \ell)=\frac{1}{N_{\vec{\sigma}}^{X, \Phi}(\mathbb{R}^d)}
\cdot \Big(\frac{\alpha_\ell}{\e^{\la \va_\ell, \vsg\ra}}\Big)^r, 
\qquad r \in \mathbb{N}, \, \ell=1, 2, \dots, m,
$$
which implies that each sequence $\{r\beta(r, \ell)\}_{r=1}^\infty,\, \ell=1, 2, \dots, m$, 
is a geometric sequence. 
Moreover, its common ratio is estimated as 
$$
\Big|\frac{\alpha_\ell}{\e^{\la \va_\ell, \vsg\ra}}\Big| \leq \frac{1}{\e^{\la \va_\ell, \vsg\ra}} < 1,
\qquad \ell=1, 2, \dots, m,
$$
due to $\min\{\la \vec{a}_\ell, \vec{\sigma}\ra \, | \, \ell=1, 2, \dots, m\}>0$. 
Conversely, we assume that each positive real number
$\beta(r, \ell), \, r \in \mathbb{N},\, \ell=1, 2, \dots, m$, is given by
$$
r\beta(r, \ell)=B_\ell \cdot A_\ell^{r-1}, \qquad r \in \mathbb{N}, \, \ell=1, 2, \dots, m,
$$
for some $B_\ell >0$ and some $A_\ell$ with $0 < A_\ell < 1$.
Then we can choose $0<\alpha_\ell \leq 1$ and $t_\ell > 1$ such that 
$$
0<\frac{\alpha_\ell}{\e^{t_\ell}} = A_\ell <1
$$
for $\ell=1, 2, \dots, m$. 
Let us consider a system of linear equations 
$$
\la \va_1, \vec{x} \ra=t_1, \, \la \va_2, \vec{x} \ra=t_2, \dots,
\la \va_m, \vec{x} \ra=t_m.
$$
Since $\va_1, \va_2, \dots, \va_m$ are linearly independent by {\bf (LI)}, 
the above system has a unique solution $\vec{\sigma} \in \mathbb{R}^d$ which satisfies that 
$\min\{\la \va_\ell, \vec{\sigma} \ra \, | \, \ell=1, 2, \dots, m\}>0$. Then we have 
$$
\beta(r, \ell)=\frac{1}{r}B_\ell \cdot A_\ell^{r-1}=\frac{B_\ell  \e^{\la \va_\ell, \vsg\ra}}{\alpha_\ell} \cdot
\frac{\alpha_\ell^r \e^{-r\la \va_\ell, \vsg\ra}}{r}, \qquad r \in \mathbb{N}, \, \ell=1, 2, \dots, m.
$$
Thus, the proof is completed 
by taking a Poisson random variable $K$ with a parameter 
$N_{\vsg}^{X, \Phi}(\mathbb{R}^d)>0$ independent of $\{\mathcal{X}_n\}_{n=1}^\infty$ and by putting 
$\Phi(\mathcal{Y})=\mathcal{X}_1+\mathcal{X}_2+\cdots+\mathcal{X}_K$. 
\end{proof}

\subsection{{\bf Relation with multidimensional Shintani zeta functions}}

As the Riemann zeta function \eqref{Euler} has both the Euler product and 
the Dirichlet series representation, each multidimensional finite Euler product 
$Z_{f\!E}^{X, \Phi} \in \mathcal{Z}_{f\!E}^{X, \Phi}$ is also expected to have
the series representation. 

We recall the following elementary relation between 
coefficients of Euler products and those of Dirichlet series expansions.

\begin{pr}[cf.\,{\cite[Lemma 2.2]{Steuding}}]\label{Steuding's prop}
Let $q \in \mathbb{N}$. Suppose that a function $\mathcal{L}=\mathcal{L}(s)$ is represented as
$$
\mathcal{L}(s)=\sum_{n=1}^\infty \frac{A(n)}{n^s}=\prod_{p \in \mathbb{P}}
\prod_{j=1}^q(1-\alpha_j(p)p^{-s})^{-1}, \qquad \mathrm{Re}\,s>1, 
$$
where $A(n) \in \mathbb{C}, \, n \in \mathbb{N},$ and 
$\alpha_j(p) \in \mathbb{C}, \, j=1, 2, \dots, q, \, p \in \mathbb{P}$. 
Then $A(n)$ is multiplicative, that is, $A(mn)=A(m)A(n)$ for 
mutually prime $m, n \in \mathbb{N}$, and 
$$
A(n)=\prod_{p|n}\sum_{\substack{0 \le \theta_1, \theta_2, \dots, \theta_q 
\\ \theta_1+\theta_2+\cdots+\theta_q=\nu(n; \, p)}}\prod_{j=1}^q
\alpha_j(p)^{\theta_{j}}, \qquad n \in \mathbb{N},
$$
where $\nu(n; \, p)$ denotes the exponent of $p \in \mathbb{P}$
in the prime factorization of $n \in \mathbb{N}$. 
\end{pr}

We obtain the following representation by applying the proposition above. 

\begin{pr}
Let $Z_{f\!E}^{X, \Phi}(\vs) \in \mathcal{Z}_{f\!E}^{X, \Phi}$ be of the form \eqref{finite-Euler-1},
where $\min\{\mathrm{Re}\la \va_\ell, \vs \ra\, | \, \ell=1, 2, \dots, m\}>1$. 
Then we obtain
\begin{equation}\label{series representation}
Z_{f\!E}^{X, \Phi}(\vs)=\sum_{k_1, k_2, \dots, k_m=0}
^\infty \Big(\frac{\alpha_1}{\e^{\la \va_1, \vs\ra}}\Big)^{k_1}
\Big(\frac{\alpha_2}{\e^{\la \va_2, \vs\ra}}\Big)^{k_2} \cdots 
\Big(\frac{\alpha_m}{\e^{\la \va_m, \vs\ra}}\Big)^{k_m}
\end{equation}
and the infinite series in \eqref{series representation}
absolutely converges 
if $\min\{\mathrm{Re}\la \va_\ell, \vs \ra\, | \, \ell=1, 2, \dots, m\}>1$. 
In particular, one has $\mathcal{Z}_{f\!E}^{X, \Phi} \subset \mathcal{Z}_S^{X, \Phi}$. 
\end{pr}

\noindent
{\bf Proof.} By taking $m$ distinct $p_1, p_2, \dots, p_m \in \mathbb{P}$, one can write
$$
Z_{f\!E}^{X, \Phi}(\vs)=\prod_{p \in \mathbb{P}}\prod_{\ell=1}^m
 (1-\alpha_\ell(p) p^{-\la \vec{b}_\ell, \vs\ra})^{-1}, 
$$
where $\alpha_\ell(p) \in \mathbb{C}$ is given by \eqref{p-reduction} and we put 
$\vec{b}_\ell=(\log p_\ell)^{-1}\va_\ell, \, \ell=1, 2, \dots, m$. 
Then Proposition \ref{Steuding's prop} in the case where $q=1$ yields
\begin{equation}\label{Eq1}
\prod_{p \in \mathbb{P}}
 (1-\alpha_\ell(p) p^{-\la \vec{b}_\ell, \vs\ra})^{-1}
 =\prod_{p \in \mathbb{P}}\Big(1+\sum_{r=1}^\infty \alpha_\ell(p)^r p^{-r\la \vb_\ell, \vs\ra}\Big)
 =\sum_{n_\ell=1}^\infty \frac{a_\ell(n_\ell)}{n_\ell^{\la \vb_\ell, \vs\ra}},
\end{equation}
for $\ell=1, 2, \dots, m$, 
where the coefficient $a_\ell(n_\ell), \, \ell=1, 2, \dots, m$ is given by
$$
a_\ell(n_\ell)=\prod_{p|n_\ell}\alpha_\ell(p)^{\nu(n_\ell; \, p)}, \qquad n_\ell \in \mathbb{N}.
$$
On the other hand, it follows from \eqref{p-reduction} that 
\begin{equation}\label{Eq2}
a_\ell(n_\ell)=\begin{cases}
\alpha_\ell^{k_\ell} & \text{if }n_\ell=p_\ell^{k_\ell}, \, k_\ell =0, 1, 2, \dots\\
0 & \text{otherwise}
\end{cases},
\end{equation}
where we understand $0^0=1$ if needed. 
It is obvious that the function $a_\ell, \, \ell=1, 2, \dots, m$, is multiplicative
with $|a_\ell(n_\ell)|=O(n_\ell^\ve)$ for any $\ve>0$ by definition. 
Then we obtain
$$
\begin{aligned}
Z_{f\!E}^{X, \Phi}(\vs)&=\prod_{\ell=1}^m\sum_{k_\ell=0}^\infty
\frac{a_\ell(p_\ell^{k_\ell})}{p_\ell^{k_\ell \la \vb_\ell, \vs\ra}}
=\prod_{\ell=1}^m\sum_{k_\ell=0}^\infty \frac{\alpha_\ell^{k_\ell}}{\e^{k_\ell \la \va_\ell, \vs \ra}}\\
&=\sum_{k_1, k_2, \dots, k_m=0}^\infty \Big(\frac{\alpha_1}{\e^{\la \va_1, \vs\ra}}\Big)^{k_1}
\Big(\frac{\alpha_2}{\e^{\la \va_2, \vs\ra}}\Big)^{k_2} \cdots 
\Big(\frac{\alpha_m}{\e^{\la \va_m, \vs\ra}}\Big)^{k_m}
\end{aligned}
$$
by combining \eqref{Eq1} with \eqref{Eq2}. Moreover, in view of 
$|\alpha_\ell| \leq 1, \, \ell=1, 2, \dots, m$, and 
$\min\{\mathrm{Re}\la \va_\ell, \vs \ra\, | \, \ell=1, 2, \dots, m\}>1$, one has
$$
\sum_{k_\ell=0}^\infty \Big|\frac{\alpha_\ell}{\e^{\la \va_\ell, \vs\ra}}\Big|^{k_\ell}
\leq \sum_{k_\ell=0}^\infty \frac{1}{\e^{k_\ell\la \va_\ell, \vs\ra}}
\leq \sum_{k_\ell=0}^\infty \frac{1}{\e^{k_\ell}} =\frac{\e}{\e-1}<+\infty
$$
for $\ell=1, 2, \dots, m$, which implies that the infinite series in \eqref{series representation} 
is absolutely convergent 
when $\min\{\mathrm{Re}\la \va_\ell, \vs \ra\, | \, \ell=1, 2, \dots, m\}>1$. 
This completes the proof. \qed

\subsection{Infinite range random walks generated by multidimensional finite Euler products}
We define random walks on crystal lattices whose range
is infinite in this subsection. 
Compared with the finite range cases, the random walks 
on crystal lattices have not been studied so much due to 
the fact that there have been few treatable multivariate 
functions corresponding to Fourier transformations of 
such random walks.
We define a class of certain infinite range random walks
on crystal lattices generated by multidimensional finite Euler 
products. We take 
$\va_1, \va_2, \dots, \va_m \in \mathcal{J}$ and $\vsg \in \mathbb{R}^d$. 
Let $x_* \in V$ be an origin of $X$.
Then, we define a infinite range random walk on $X$ starting from $x_*$
which is generated by multidimensional finite Euler products 
as in the following.

\begin{df}[infinite range random walks on $X$
generated by a multidimensional finite Euler product]
Suppose that $m$ vectors 
$\va_1, \va_2, \dots, \va_m$ satisfy {\bf (LI)} and $\vsg$ satisfies 
$\min\{\la \va_\ell, \vsg\ra| \, \ell=1, 2, \dots, m\}>0$.
Let $Z_{f\!E}^{X, \Phi} \in \mathcal{Z}_{f\!E}^{X, \Phi}$ and 
$\mathcal{Y}_{\vsg}$ be a compound Poisson random variable on $X$
generated by $Z_{f\!E}^{X, \Phi}$. 
Then, a sequence of $V$-valued independent random variables $\{\mathcal{W}_n\}_{n=0}^\infty$ is called an 
infinite range random walk 
generated by a multidimensional finite Euler product on $X$ if 
$\mathcal{W}_0=x_*$ a.s.\,and the characteristic function
of 
$\Phi(\mathcal{W}_n), \, n \in \mathbb{N}$, is given by 
$$
{\bf{E}}[\e^{\ii \la \vt, \Phi(\mathcal{W}_n) \ra}]
=
\big(f_{\vsg}^{X, \Phi}(\vt)\big)^n, \qquad t \in \mathbb{R}^d. 
$$
\end{df}

Note that our random walk in this case is also defined by its characteristic function as in Definition \ref{Def:RW-Shintani}.
Similarly, it has an independent increments but they do not depend on which $x\in V$ they are at.
So that it seems to be more simply defined compared to the previous case.
The Euler product representation makes us visible whether they are compound Poisson (i.e., infinitely divisible) or not and obtain their L\'evy measures easily when they are to be so.

\section{{\bf Examples}}
\label{Sec:Example}
We give some comprehensible examples of crystal lattices and multiple zeta functions on them in this section. Moreover, random walks on crystal lattices generated by these zeta functions are also given. 
We start with the case where $d=1$. 

\begin{ex}\normalfont
Let $\Gamma=\mathbb{Z}=\la \gamma \ra$, $V=\mathbb{Z}$ and 
$E=\{e=\{x, y\} \in \mathbb{Z} \times \mathbb{Z} \, | \, y-x=\pm1\}.$
The $\Gamma$-action on $X=(V, E)$ is given by 
$\gamma x:=x+1$ for $x \in V$. 
Then one can see that the graph $X$ is a $\Gamma$-covering graph of a
{\rm 1}-bouquet graph $X_0=(V_0, E_0)$, where $V_0=\{\x\}$ and 
$E_0=\{e, \ol{e}\}$ (see Figure \ref{line}).

\begin{figure}[ht]
\begin{center}
\includegraphics[width=9cm]{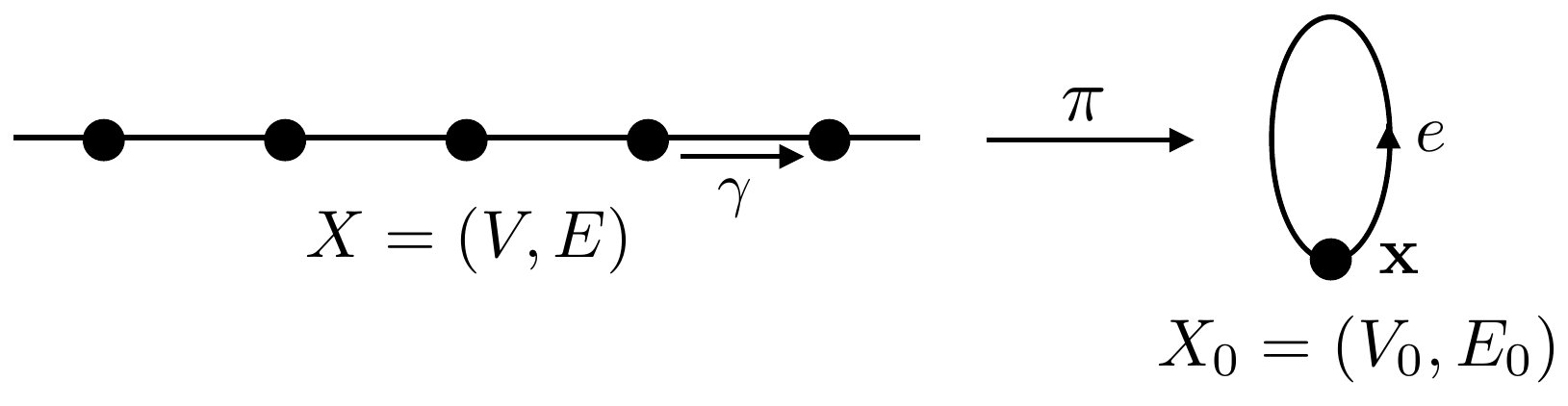}
\end{center}
\caption{the 1-dimensional crystal lattice and its quotient graph.}
\label{line}
\end{figure}

A periodic realization $\Phi : V \to \mathbb{R}$ is naturally defined by
$\Phi(0)=0$ and $\Phi(1)=1$. 
We identify $\gamma \in \Gamma$ with $1 \in \mathbb{R}$ 
and put 
$\mathcal{J}=\{c=ka \, | \, k \in \mathbb{R}, \, a \in \Phi(V)=\mathbb{Z}\}$.

Let $m=r=2$, $\lambda_{\ell j}=\delta_{\ell j}, \, u_j=1, \, \ell, j=1, 2$
and $c_1, c_2 \in \mathcal{J}$. We consider a multidimensional Shintani zeta function defined by 
$$
Z_S^{X, \Phi}(s):=\sum_{n_1, n_2=0}^\infty 
\frac{\theta(n_1, n_2)}{(n_1+1)^{c_1s}(n_2+1)^{c_2s}}, 
\qquad s \in \mathbb{C}. 
$$

\vspace{1mm}
\noindent
{\bf (i)} 
Let $j_1, j_2$ be distinct positive integers greater than 1. 
Suppose that $c_1=-(\log j_1)^{-1}$, $c_2=(\log j_2)^{-1}$ and 
$$
\theta(n_1, n_2)=\begin{cases}
\alpha \e^{-\sigma} & \text{if }(n_1, n_2)=(j_1-1, 0) \\
\beta \e^{\sigma} & \text{if }(n_1, n_2)=(0, j_2-1) \\
0 & \text{otherwise}
\end{cases}, 
$$
where $\alpha, \beta \ge 0$, $\alpha+\beta=1$ 
and $\sigma>1$.
Then we have
$$
Z_S^{X, \Phi}(s)=\frac{\alpha \e^{-\sigma}}
{j_1^{-(\log j_1)^{-1}s}}+
\frac{\beta \e^{\sigma}}{j_2^{(\log j_2)^{-1}s}}
=\alpha \e^{s-\sigma}+\beta \e^{-s+\sigma}
$$
Therefore, the characteristic function $f_\sigma^{X, \Phi}(t)$ is given by
$$
f_\sigma^{X, \Phi}(t)=
\frac{Z_S^{X, \Phi}(\sigma+\ii t)}{Z_S^{X, \Phi}(\sigma)}
=\frac{\alpha \e^{\ii t}+\beta\e^{-\ii t}}{\alpha+\beta}
=\alpha\e^{\ii t}+\beta \e^{-\ii t}, \qquad t \in \mathbb{R},
$$
which is that of $\mathbb{Z}$-valued random variable $\mathcal{Y}$
given by ${\bf P}(\mathcal{Y}=1)=\alpha$ 
and ${\bf P}(\mathcal{Y}=-1)=\beta$. 
This clearly induces the usual nearest-neighbor random walk on $\mathbb{Z}$ and also simple one when $\alpha=\beta=1/2$. 
Note that the random variable $\mathcal{Y}$ is not infinitely divisible. 

\vspace{2mm}
\noindent
{\bf (ii)} Let $\lambda>0$ and $j>1$ be an integer. 
Suppose that $c_1=-(\log j)^{-1}$, $c_2=0$ and 
$$
\theta(n_1, n_2)=\begin{cases}
1 & \text{if }(n_1, n_2)=(0, 0) \\
\lambda^k \e^{-k\sigma}/k! & \text{if }(n_1, n_2)=(j^k-1, 0), 
\, k=1, 2, \dots \\
0 & \text{otherwise}
\end{cases}, 
$$
where $\sigma>1$. Then we have
$$
Z_S^{X, \Phi}(s)=\sum_{k=0}^\infty 
\frac{\lambda^k \e^{-k\sigma}/k!}{(j^k)^{-(\log j)^{-1}s}}
=\sum_{k=0}^\infty \frac{(\lambda \e^{s-\sigma})^k}{k!}
=\exp(\lambda \e^{s-\sigma}). 
$$
Therefore, we obtain
$$
f_\sigma^{X, \Phi}(t)=
\frac{Z_S^{X, \Phi}(\sigma+\ii t)}{Z_S^{X, \Phi}(\sigma)}
=\frac{\exp(\lambda \e^{\ii t})}{\exp(\lambda)}
=\exp\big(\lambda(\e^{\ii t}-1)\big), \qquad t \in \mathbb{R},
$$
which corresponds to the usual Poisson random variable $\mathcal{Y}$
given by 
${\bf P}(\mathcal{Y}=k)=\e^{-\lambda}\lambda^k/k!$
for $k=0, 1, 2, \dots$. 
Note that the random variable $\mathcal{Y}$ is infinitely divisible. 
\end{ex}


In the following, we discuss three typical examples of 
2-dimensional crystal lattices and multiple zeta functions on them. 

\begin{ex}[square lattice]\normalfont
Let $d=2$ and $\Gamma= \la \gamma_1, \gamma_2 \ra$. 
We define $V=\mathbb{Z}^2$ and 
$$
E:=\big\{e=\{x, y\} \in \mathbb{Z}^2 \times \mathbb{Z}^2 \, | \, 
y-x \in \{(\pm 1, 0), (0, \pm 1)\}\big\}.
$$
The $\Gamma$-action on $X=(V, E)$ is given by 
$\gamma_1x:=x+(1, 0)$ and 
$\gamma_2x:=x+(0, 1)$ for $x \in V$. 
Then we easily see that the infinite graph $X$
is a $\Gamma$-covering of a 2-{\it bouquet} graph 
$X_0=(V_0, E_0)$, where $X_0=\{\bf{x}\}$ and $E_0=\{e_1, e_2, \ol{e}_1, \ol{e}_2\}$
(see Figure \ref{square}). 
The first homology group $\h_1(X_0, \mathbb{R})$
is given by $\{a_1[e_1]+a_2[e_2] \, | \, a_1, a_2 \in \mathbb{R}\},$ where $[e_1]$ and $[e_2]$ are 
homotopy classes of $e_1$ and $e_2$, respectively. 
Therefore, $X$ is maximal abelian. 

\begin{figure}[ht]
\begin{center}
\includegraphics[width=9cm]{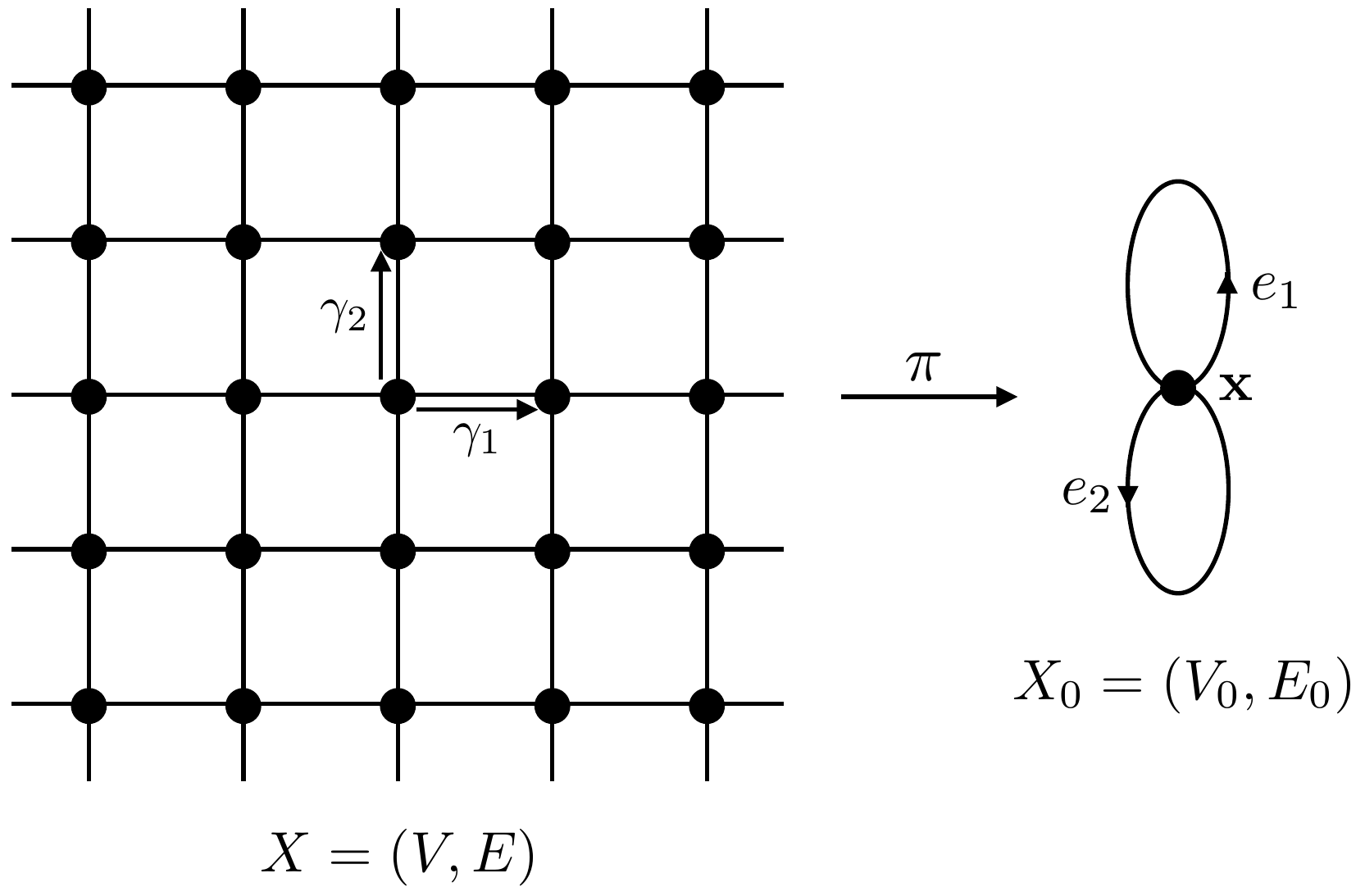}
\end{center}
\caption{the 2-dimensional square lattice and its quotient graph.}
\label{square}
\end{figure}

A periodic realization $\Phi : V \to \mathbb{R}^2$ is defined by
$\Phi\big((0, 0)\big)=(0, 0)$ and 
$$
\Phi\big((1, 0)\big)=(1, 0), \qquad \Phi\big((0, 1)\big)=(0, 1).
$$ 
We identify $\gamma_1, \gamma_2 \in \Gamma$
as vectors $(1, 0), \, (0, 1) \in \mathbb{R}^2$, respectively. 
We also define the sets of vectors $\mathcal{J}_S$ and $\mathcal{J}_E$ by
$$
\mathcal{J}_S=\{\vc=k\va \, | \, k \in \mathbb{R}, \, \va \in \mathbb{Z}^2\}, \quad
\mathcal{J}_E=\big\{\va= k_1\gamma_1+k_2\gamma_2\, \big| \, k_1, k_2 \in \mathbb{Z}\big\}. 
$$

We consider a multidimensional Shintani zeta function and
a multidimensional finite Euler product on $X$ 
given by 
\begin{align}
Z_S^{X, \Phi}(\vs)&:=\sum_{n_1, n_2, n_3, n_4=0}^\infty
\frac{\theta(n_1, n_2, n_3, n_4)}
{\prod_{\ell=1}^4 \big(\sum_{j=1}^4 \lambda_{\ell j}(n_j+u_j)\big)^{\la \vc_\ell, \vs\ra}}, \qquad \vs \in \mathbb{C}^2, 
\label{square Shintani}\\
Z_{f\!E}^{X, \Phi}(\vs)&:=(1-\alpha_1 \e^{-\la \va_1, \vs \ra})^{-1}
(1-\alpha_2 \e^{-\la \va_2, \vs \ra})^{-1}, \qquad \vs \in \mathbb{C}^2, 
\label{square Euler}
\end{align}
where $\vc_1, \vc_2, \vc_3, \vc_4 \in \mathcal{J}_S$ and 
$\va_1, \va_2 \in \mathcal{J}_E$. 

\vspace{2mm}
\noindent
{\bf (i)} Let $j_1, j_2, j_3, j_4$ be four distinct positive integers greater than 1. 
Suppose that $\lambda_{\ell j}=\delta_{\ell j}$, $u_j=1$ for $\ell, j=1, 2, 3, 4$, 
$\vc_1=(-(\log j_1)^{-1}, 0)$, 
$\vc_2=((\log j_2)^{-1}, 0)$, 
$\vc_3=(0, -(\log j_3)^{-1})$,
$\vc_4=(0, (\log j_4)^{-1})$ and 
$$
\theta(n_1, n_2, n_3, n_4)=\begin{cases}
\alpha \e^{-\sigma_1} & 
\text{if }(n_1, n_2, n_3, n_4)=(j_1-1, 0, 0, 0) \\
\alpha \e^{\sigma_1} & 
\text{if }(n_1, n_2, n_3, n_4)=(0, j_2-1, 0, 0) \\
\beta \e^{-\sigma_2} & 
\text{if }(n_1, n_2, n_3, n_4)=(0, 0, j_3-1, 0)\\
\beta' \e^{\sigma_2} & 
\text{if }(n_1, n_2, n_3, n_4)=(0, 0, 0, j_4-1)\\
0 & \text{otherwise}
\end{cases}
$$
as in \eqref{square Shintani}, where $\alpha, \alpha', \beta, \beta' \ge 0$, $\alpha+\alpha'+\beta+\beta'=1$
and $\vsg=(\sigma_1, \sigma_2) \in \mathbb{R}^2 \setminus \{\bm{0}\}$. 
Then we have 
$$
\begin{aligned}
Z_S^{X, \Phi}(\vs)&=
\frac{\alpha \e^{-\sigma_1}}{j_1^{\la \vc_1, \vs \ra}}+
\frac{\alpha' \e^{\sigma_1}}{j_2^{\la \vc_2, \vs \ra}}+
\frac{\beta \e^{-\sigma_2}}{j_3^{\la \vc_3, \vs \ra}}+
\frac{\beta' \e^{\sigma_2}}{j_4^{\la \vc_4, \vs \ra}}\\
&=\alpha \e^{\la (1, 0), \vs \ra - \sigma_1}
+\alpha' \e^{\la (-1, 0), \vs \ra + \sigma_1}
+\beta \e^{\la (0, 1), \vs \ra - \sigma_2}
+\beta' \e^{\la (0, -1), \vs \ra - \sigma_2}. 
\end{aligned}
$$
By putting $\vs=\vsg+\ii \vt=(\sigma_1, \sigma_2)+\ii (t_1, t_2)$
for $\vt=(t_1, t_2) \in \mathbb{R}^2$, we obtain
$$
f_{\vsg}^{X, \Phi}(\vt)
=\frac{Z_S^{X, \Phi}(\vsg+\ii \vt)}{Z_S^{X, \Phi}(\vsg)}
=\alpha \e^{\ii t_1}+\alpha'\e^{-\ii t_1}
+\beta \e^{\ii t_2}+\beta'\e^{-\ii t_2}, 
$$
which is a characteristic function of a 
$\mathbb{Z}^2$-valued random variable $\mathcal{Y}$ given by
$$
\begin{aligned}
{\bf P}(\mathcal{Y}=(1, 0))&=\alpha, \quad
&{\bf P}(\mathcal{Y}=(-1, 0))&=\alpha',\\
{\bf P}(\mathcal{Y}=(0, 1))&=\beta, \quad 
&{\bf P}(\mathcal{Y}=(0, -1))&=\beta'. 
\end{aligned}
$$
This random variable induces a $\mathbb{Z}^2$-valued nearest-neighbor random walk
and also a simple one when $\alpha=\alpha'=\beta=\beta'=1/4$. 

\vspace{1mm}
\noindent{\bf (ii)} 
Suppose that $\alpha_1=\alpha_2=1/3$, 
$\va_1=(1, 0)$ and $\va_2=(0, 1)$ as in \eqref{square Euler}. 
Then we have 
$$
Z_{f\!E}^{X, \Phi}(\vs)=\frac{3}{3-\e^{\la (1, 0), \vs \ra}}\cdot
\frac{3}{3-\e^{\la (0, 1), \vs \ra}}, \qquad \vs \in \mathbb{C}^2. 
$$
It readily follows from Theorem \ref{New-zeta-1} that, 
for $\vsg:=(\log 2, \log 2) \in \mathbb{R}^2$, 
the function 
$$
f_{\vsg}^{X, \Phi}(\vs)=
\frac{Z_{f\!E}^{X, \Phi}(\vsg+\ii \vt)}{Z_{f\!E}^{X, \Phi}(\vsg)}
=\frac{1}{(3-2\e^{\ii t_1})(3-2\e^{\ii t_2})}
, \qquad \vt \in \mathbb{R}^2,
$$
is a compound Poisson characteristic function whose 
finite L\'evy measure on $\mathbb{R}^2$ is given by
$$
N_{\vsg}^{X, \Phi}(d\vx)
=\sum_{r=1}^\infty \frac{1}{r}\Big(\frac{2}{3}\Big)^r 
\big(\delta_{(r, 0)}(d\vx)+\delta_{(0, r)}(d\vx)\big). 
$$
Note that we have $N_{\vsg}^{X, \Phi}(\mathbb{R}^2)=2\log{3}$.
This clearly induces a compound Poisson random variable 
$\Phi(\mathcal{Y}_{\vsg})=\sum_{k=1}^{K_{\vsg}^{X, \Phi}}\mathcal{X}_{k}$,
where $\{\mathcal{X}_k\}_{k=1}^\infty$ is a family of 
$\mathbb{R}^2$-valued independent and identically distributed
random variables whose common distribution is given by
$$
{\bf P}\big(\mathcal{X}_1=(r, 0)\big)
={\bf P}\big(\mathcal{X}_1=(0, r)\big)
=\frac{1}{r}\Big(\frac{2}{3}\Big)^r \cdot \frac{1}{2\log 3}, \qquad r \in \mathbb{N},
$$
and $K_{\vsg}^{X, \Phi}$ is a Poisson random variable with a
parameter $2\log{3}$ independent of $\{\mathcal{X}_k\}_{k=1}^\infty$. 
The $n$-fold convolution of the distribution of the random variable $\Phi(\mathcal{Y})$ induces an infinite range random walk $\{\mathcal{W}_n\}_{n=0}^\infty$ generated by $Z_{f\!E}^{X, \Phi}$
whose characteristic function is
$$
\mathbb{E}[\e^{\ii \la \vt, \Phi(\mathcal{W}_n) \ra}]
=\big(f_{\vsg}^{X, \Phi}(\vt)\big)^n
=\frac{1}{(3-2\e^{\ii t_1})^n(3-2\e^{\ii t_2})^n}, \qquad 
n \in \mathbb{N}, \, \vt=(t_1, t_2) \in \mathbb{R}^2, 
$$
with $\Phi(\mathcal{W}_0)=\bm{0}$. 
\end{ex}

\begin{ex}[triangular lattice] \normalfont 
Let $d=2$ and $\Gamma=\la \gamma_1, \gamma_2\ra$. 
We define $V=\mathbb{Z}^2$ and 
$$
E:=\big\{e=\{x, y\} \in \mathbb{Z}^2 \times \mathbb{Z}^2 \, | \,
y - x \in \{(\pm 1, 0), (0, \pm1), (\pm1, \mp1)\}\big\}. 
$$
The $\Gamma$-action on $X=(V, E)$ is given by 
$\gamma_1x:=x+(1, 0)$ and 
$\gamma_2x:=x+(0, 1)$ for $x \in V$. 
Then we see that $X$
is a $\Gamma$-covering of a 3-bouquet graph 
$X_0=(V_0, E_0)$, where $X_0=\{\bf{x}\}$ and 
$E_0=\{e_1, e_2, e_3, \ol{e}_1, \ol{e}_2, \ol{e}_3\}$
(see Figure \ref{triangular}). 
Since the first homology group $\h_1(X_0, \mathbb{R})$
is of dimension 3, it is known that $X$ is not maximal abelian. 

\begin{figure}[htb]
\begin{center}
\includegraphics[width=9cm]{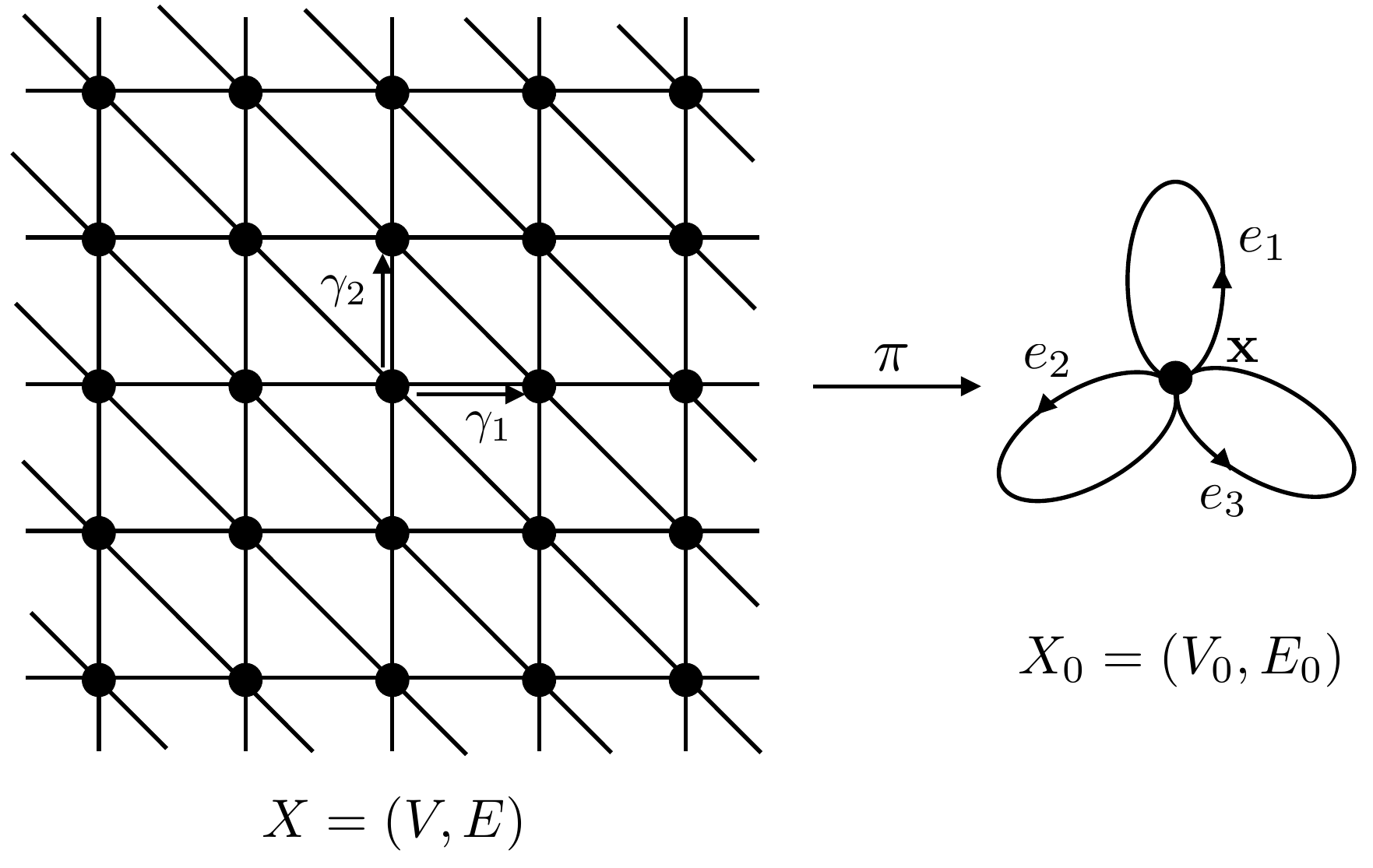}
\end{center}
\caption{the triangular lattice and its quotients graph.}
\label{triangular}
\end{figure}

We take a periodic realization $\Phi : V \to \mathbb{R}^2$
and the set of vectors $\mathcal{J}_S$
as in the previous example. 
Let $N \in \mathbb{N}$ and 
$$
\mathcal{M}:=\{\vk=(k_1, k_2) 
\in \mathbb{Z}^2 \, : \, 
|k_1+k_2|\le N, \, |k_1|, |k_2| \le N\},
$$ 
which is denoted by $\mathcal{M}=\{\vk_\ell\}_{\ell=1}^{M}$
for a notational convention, where 
$M:=|\mathcal{M}|=3N^2+3N+1$. 
See Figure \ref{visiting vertices} below.

\begin{figure}[htb]
\begin{center}
\includegraphics[width=6cm]{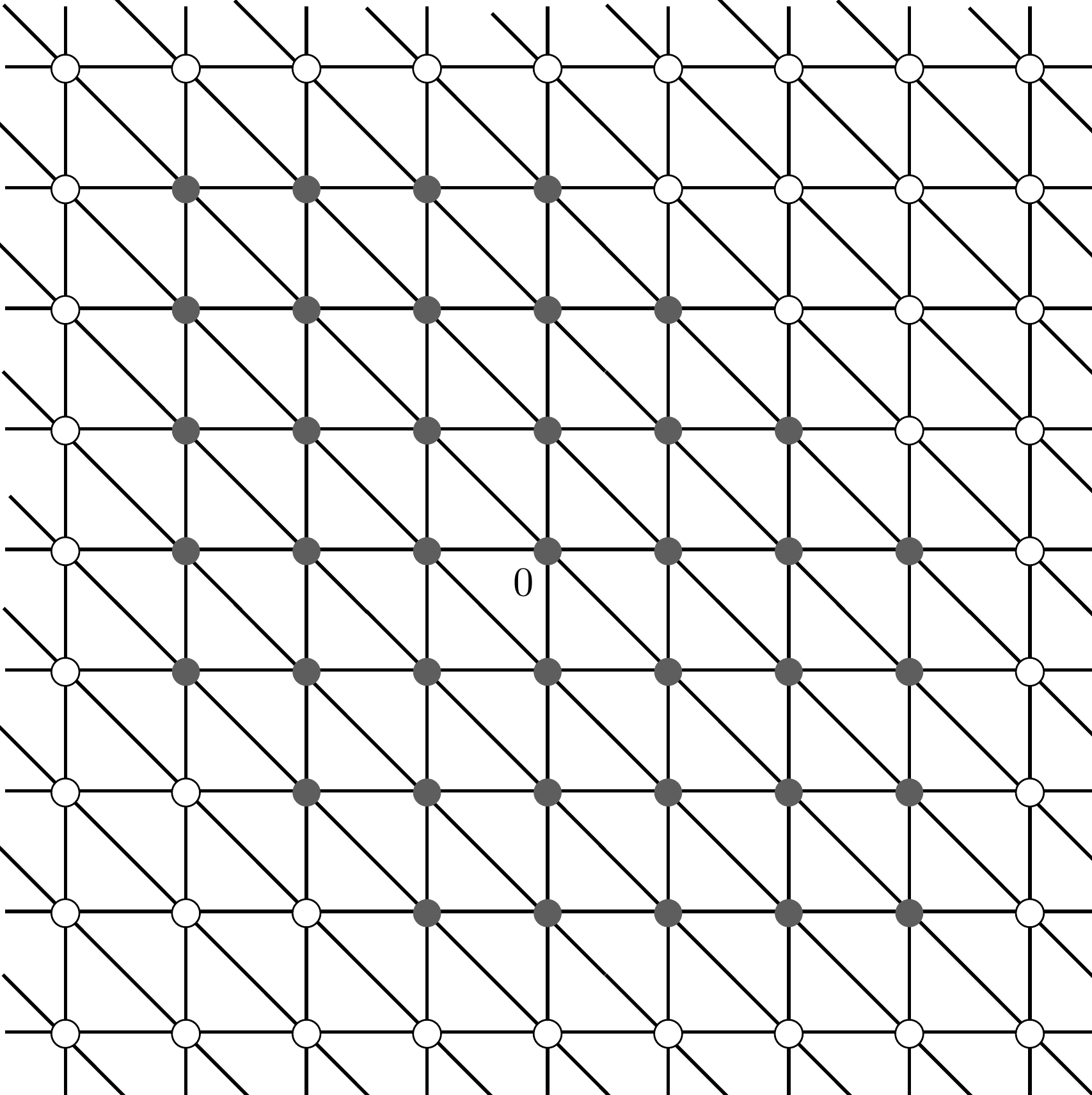}
\end{center}
\caption{the set $\mathcal{M}$ in the case where $N=3$.}
\label{visiting vertices}
\end{figure}

Let us consider a multidimensional Shintani zeta function on $X$ 
given by
\begin{align}
Z_S^{X, \Phi}(\vs)&:=\sum_{n_1, n_2, \dots, n_M=0}^\infty
\frac{\theta(n_1, n_2, \dots, n_M)}
{\prod_{\ell=1}^M (n_\ell+1)^{\la \vc_\ell, \vs\ra}}, \qquad \vs \in \mathbb{C}^2, 
\label{triangular Shintani}
\end{align}
where 
$\vc_\ell \in \mathcal{J}_S$ for $\ell=1, 2, \dots, M$.
Let $\alpha(\vk), \, \vk \in \mathcal{M}$ be 
real numbers taking values in the interval $(0, 1)$ satisfying $\sum_{\vk \in \mathcal{M}}\alpha(\vk)=1$
and $j_\ell, \, \ell=1, 2, \dots, M$, be distinct positive integers greater than 1. 
Suppose that $\vc_\ell=-(\log j_\ell)^{-1}\vk_{\ell},\, \ell=1, 2, \dots, M$, and
$$
\begin{aligned}
&\theta(n_1, n_2, \dots, n_M)\\
&=\begin{cases}
\alpha(\vk_\ell)\e^{-\la \vk_\ell, \vsg \ra} & \text{if }
(n_1, \dots, n_\ell, \dots, n_M)
=(0, \dots, 0, j_\ell-1, 0, \dots, 0), \, \ell=1, 2, \dots, M\\
0 & \text{otherwise}
\end{cases},
\end{aligned}
$$
where $\vsg \in \mathbb{R}^2 \setminus \{\bm{0}\}$. 
Then we have
$$
Z_S^{X, \Phi}(\vs)=
\sum_{\ell=1}^M \frac{\alpha(\vk_\ell)
\e^{-\la \vk_\ell, \vsg\ra}}{j_\ell^{\la \vc_\ell, \vs\ra}}
=\sum_{\vk \in \mathcal{M}}\alpha(\vk)\e^{\la \vk, \vs-\vsg\ra} 
$$
and 
$$
f_{\vsg}^{X, \Phi}(\vt)
=\frac{Z_S^{X, \Phi}(\vsg+\ii\vt)}{Z_S^{X, \Phi}(\vsg)}
=\sum_{\vk \in \mathcal{M}}\alpha(\vk)\e^{\ii\la \vk, \vt\ra}, 
\qquad \vt \in \mathbb{R}^2.
$$
Consequently, it is known that the $\mathbb{R}^2$-valued random
variable $\mathcal{Y}$ whose characteristic function is $f_{\vsg}^{X, \Phi}$
is represented as
${\bf P}(\mathcal{Y}=\vk)=\alpha(\vk), \, \vk \in \mathcal{M}$. 
This means that the random walk $\{\mathcal{W}_n\}_{n=0}^\infty$ 
generated by $Z_S^{X, \Phi}$ can visit all vertices whose 
graph distance from the current position is less than or equal to $N$
at each step. 

\end{ex}

\begin{ex}[hexagonal lattice] \normalfont

Let $d=2$ and $\Gamma=\la \gamma_1, \gamma_2\ra$. 
We define
$$
\begin{aligned}
V&:=\mathbb{Z}^2=\{x=(x_1, x_2) \, | \, x_1, x_2 \in \mathbb{Z}\},\\  
E&:=\big\{e=\{x, y\} \in \mathbb{Z}^2 \times \mathbb{Z}^2 \, | \,
y - x =(\pm 1, 0), (0, (-1)^{x_1+x_2})\big\}. 
\end{aligned}
$$
The $\Gamma$-action on $X=(V, E)$ is given by 
$\gamma_1x:=x+(1, 1)$ and 
$\gamma_2x:=x+(-1, 1)$ for $x \in V$. 
Then $X$
is a covering graph of a finite graph 
$X_0=(V_0, E_0)$ with the covering transformation group $\Gamma$, 
where $X_0=\{\bf{x}, {\bf y}\}$ and 
$E_0=\{e_1, e_2, e_3, \ol{e}_1, \ol{e}_2, \ol{e}_3\}$
(see Figure \ref{hexagonal}). 
Since the first homology group $\h_1(X_0, \mathbb{R})$
is of dimension 2, we see that $X$ is a maximal abelian
covering graph as well as the square lattice.

\begin{figure}[ht]
\begin{center}
\includegraphics[width=10cm]{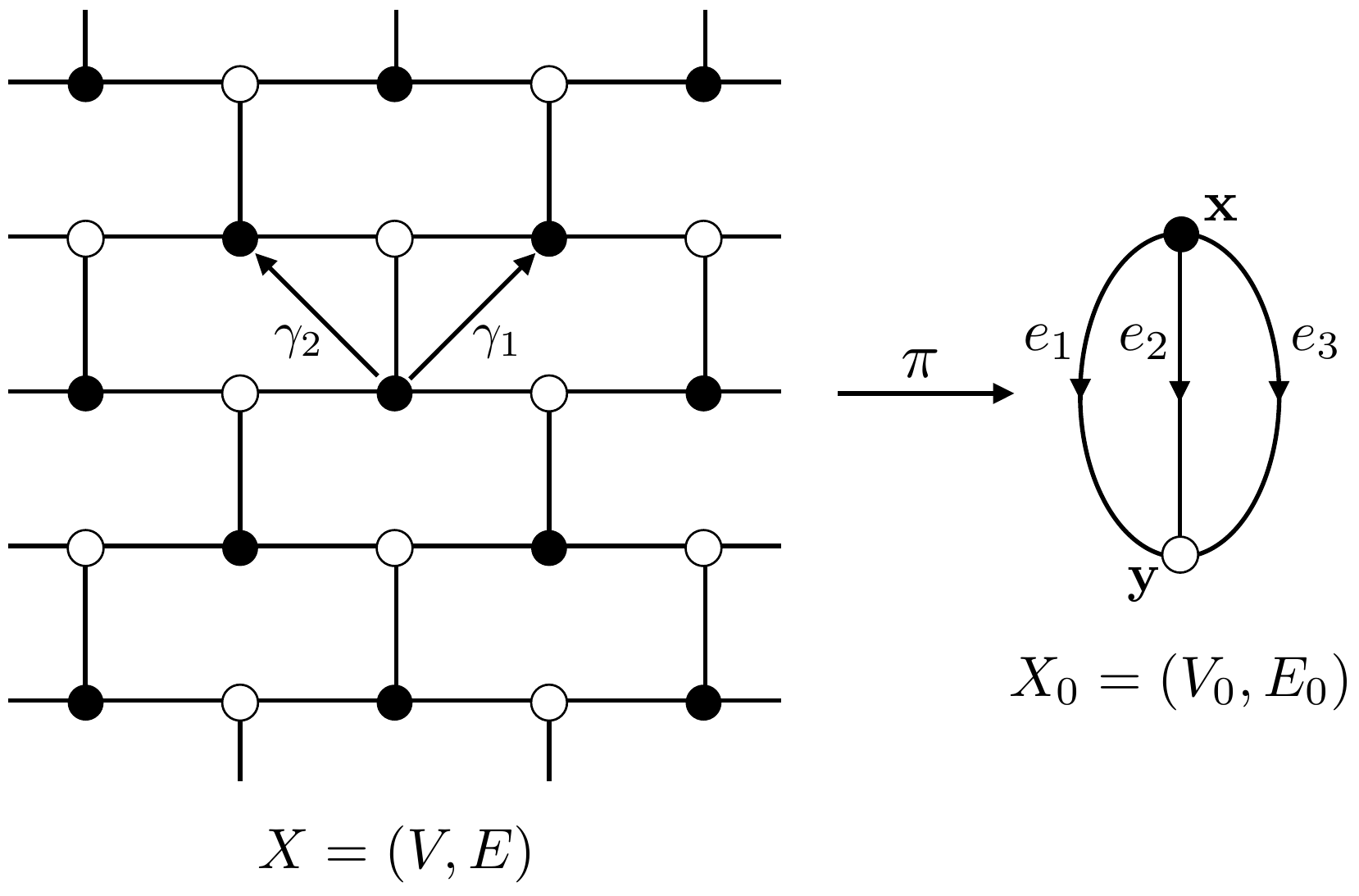}
\end{center}
\caption{the hexagonal lattice and its quotients graph.}
\label{hexagonal}
\end{figure}

\vskip5cm

A periodic realization $\Phi : V \to \mathbb{R}^2$ 
is defined as in the following. 
We put $\widetilde{x}=(0, 0)$ and $\widetilde{y}=(0, 1)$ in $V$ and let $\Phi(\widetilde{x})=\bm{0}$ and 
$\Phi(\widetilde{y})=(1/3, 2/3)$. 
Moreover, $\gamma_1, \gamma_2 \in \Gamma$ are 
identified with $(1, 0), (0, 1) \in \mathbb{R}^2$, respectively (see Figure \ref{realization of hexagonal}). 

\begin{figure}[ht]
\begin{center}
\includegraphics[width=6cm]{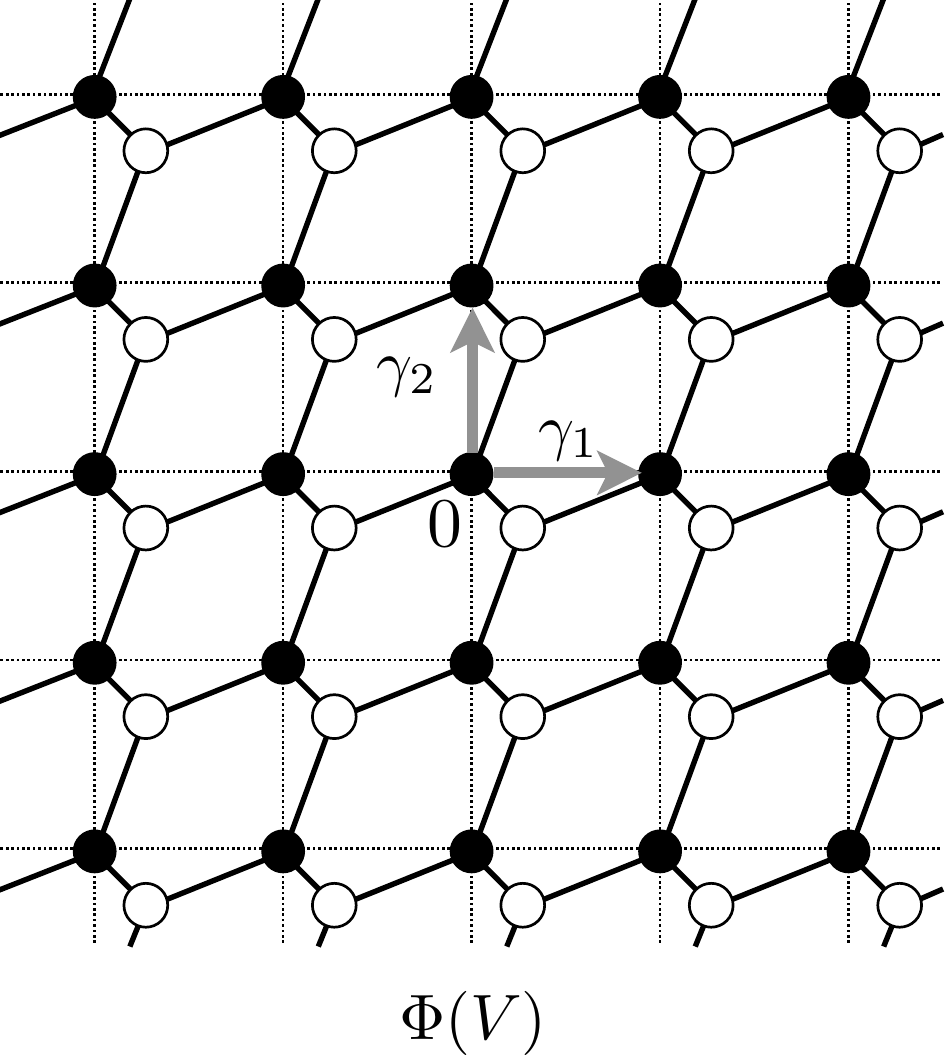}
\end{center}
\caption{a periodic realization of the hexagonal lattice.}
\label{realization of hexagonal}
\end{figure}

We define $\mathcal{J}(\Phi; {\bf x})$ and $\mathcal{J}(\Phi; {\bf y})$ by \eqref{jump set Shintani}, respectively. 
For ${\bf x}, {\bf y} \in V_0$, 
consider two multidimensional Shintani zeta functions given by
$$
\begin{aligned}
Z_S^{X, \Phi}({\bf x}, \vs):=\sum_{n_1, n_2, n_3=0}^\infty
\frac{\theta_{{\bf x}}(n_1, n_2, n_3)}{\prod_{\ell=1}^3\big(\sum_{j=1}^3 \lambda_{\ell j}(n_j+1)\big)^{\la \vc_\ell({\bf x}), \vs\ra}}, \qquad \vs \in \mathbb{C}^2,
\\
Z_S^{X, \Phi}({\bf y}, \vs):=\sum_{n_1, n_2, n_3=0}^\infty
\frac{\theta_{{\bf y}}(n_1, n_2, n_3)}{\prod_{\ell=1}^3\big(\sum_{j=1}^3 \lambda_{\ell j}(n_j+1)\big)^{\la \vc_\ell({\bf y}), \vs\ra}}, \qquad \vs \in \mathbb{C}^2,
\end{aligned}
$$
where $\vc_\ell({\bf x}) \in \mathcal{J}(\Phi; {\bf x})$
and $\vc_\ell({\bf y}) \in \mathcal{J}(\Phi; {\bf y})$
for $\ell=1, 2, 3$. 

Let $j_1, j_2, j_3$ be distinct positive integers greater than 1. 
Suppose that $\lambda_{\ell j}=\delta_{\ell j}$ for $\ell, j=1, 2, 3$, 
\begin{align*}
\vc_1({\bf x})&=\Big(-\frac{1}{3}(\log j_1)^{-1}, 
-\frac{2}{3}(\log j_1)^{-1}\Big), &
\vc_2({\bf x})&=\Big(-\frac{1}{3}(\log j_2)^{-1}, 
\frac{1}{3}(\log j_2)^{-1}\Big), \\
\vc_3({\bf x})&=\Big(\frac{2}{3}(\log j_3)^{-1}, 
\frac{1}{3}(\log j_3)^{-1}\Big), &
\vc_1({\bf y})&=\Big(\frac{1}{3}(\log j_1)^{-1}, 
\frac{2}{3}(\log j_1)^{-1}\Big), \\
\vc_2({\bf y})&=\Big(\frac{1}{3}(\log j_2)^{-1}, 
-\frac{1}{3}(\log j_2)^{-1}\Big), &
\vc_3({\bf y})&=\Big(-\frac{2}{3}(\log j_3)^{-1}, 
-\frac{1}{3}(\log j_3)^{-1}\Big)
\end{align*}
and 
$$
\theta_{{\bf x}}(n_1, n_2, n_3)
=\begin{cases}
\alpha_1 \e^{-\la (\frac{1}{3}, \frac{2}{3}), \vsg({\bf x})\ra} & 
\text{if }(n_1, n_2, n_3)=(j_1-1, 0, 0) \\
\alpha_2 \e^{-\la (\frac{1}{3}, -\frac{1}{3}), \vsg({\bf x})\ra} & 
\text{if }(n_1, n_2, n_3)=(0, j_2-1, 0) \\
\alpha_3 \e^{-\la (-\frac{2}{3}, -\frac{1}{3}), \vsg({\bf x})\ra} & 
\text{if }(n_1, n_2, n_3)=(0, 0, j_3-1) \\
0 & \text{otherwise}
\end{cases},
$$
$$
\theta_{{\bf y}}(n_1, n_2, n_3)
=\begin{cases}
\beta_1 \e^{-\la (-\frac{1}{3}, -\frac{2}{3}), \vsg({\bf y})\ra} & 
\text{if }(n_1, n_2, n_3)=(j_1-1, 0, 0) \\
\beta_2 \e^{-\la (-\frac{1}{3}, \frac{1}{3}), \vsg({\bf y})\ra} & 
\text{if }(n_1, n_2, n_3)=(0, j_2-1, 0) \\
\beta_3 \e^{-\la (\frac{2}{3}, \frac{1}{3}), \vsg({\bf y})\ra} & 
\text{if }(n_1, n_2, n_3)=(0, 0, j_3-1) \\
0 & \text{otherwise}
\end{cases},
$$
where $\alpha_1, \alpha_2, \alpha_3, \beta_1, \beta_2, \beta_3 \ge 0$, 
$\alpha_1+\alpha_2+\alpha_3=\beta_1+\beta_2+\beta_3=1$ and 
$\vsg({\bf x}), \vsg({\bf y}) \in \mathbb{R}^2 \setminus \{\bm{0}\}$. 
Then we obtain
$$
\begin{aligned}
Z_S^{X, \Phi}({\bf x}, \vs)&=\frac{\alpha_1 \e^{-\la (\frac{1}{3}, \frac{2}{3}), \vsg({\bf x})\ra}}{j_1^{\la \vc_1({\bf x}), \vs\ra}}
+\frac{\alpha_2 \e^{-\la (\frac{1}{3}, -\frac{1}{3}), \vsg({\bf x})\ra}}{j_2^{\la \vc_2({\bf x}), \vs\ra}}
+\frac{\alpha_3 \e^{-\la(-\frac{2}{3}, -\frac{1}{3}), \vsg({\bf x})\ra}}{j_3^{\la \vc_3({\bf x}), \vs\ra}}\\
&=\alpha_1\e^{\la (\frac{1}{3}, \frac{2}{3}), \vs-\vsg({\bf x})\ra}
+\alpha_2\e^{\la (\frac{1}{3}, -\frac{1}{3}), \vs-\vsg({\bf x})\ra}
+\alpha_3\e^{\la (-\frac{2}{3}, -\frac{1}{3}), \vs-\vsg({\bf x})\ra},\\
Z_S^{X, \Phi}({\bf y}, \vs)&=\frac{\beta_1 \e^{-\la (-\frac{1}{3}, -\frac{2}{3}), \vsg({\bf x})\ra}}{j_1^{\la \vc_1({\bf x}), \vs\ra}}
+\frac{\beta_2 \e^{-\la (-\frac{1}{3}, \frac{1}{3}), \vsg({\bf y})\ra}}{j_2^{\la \vc_2({\bf y}), \vs\ra}}
+\frac{\beta_3 \e^{-\la(\frac{2}{3}, \frac{1}{3}), \vsg({\bf x})\ra}}{j_3^{\la \vc_3({\bf y}), \vs\ra}}\\
&=\beta_1\e^{\la (-\frac{1}{3}, -\frac{2}{3}), \vs-\vsg({\bf y})\ra}
+\beta_2\e^{\la (-\frac{1}{3}, \frac{1}{3}), \vs-\vsg({\bf y})\ra}
+\beta_3\e^{\la (\frac{2}{3}, \frac{1}{3}), \vs-\vsg({\bf y})\ra}.
\end{aligned}
$$
Therefore, one has
$$
\begin{aligned}
f_{\vsg({\bf x})}(\vt)&=\frac{Z_S^{X, \Phi}(\vsg({\bf x})+\ii \vt)}
{Z_S^{X, \Phi}(\vsg({\bf x}))}=\alpha_1\e^{\ii(\frac{1}{3}t_1+\frac{2}{3}t_2)}
+\alpha_2\e^{\ii(\frac{1}{3}t_1-\frac{1}{3}t_2)}
+\alpha_3\e^{\ii(-\frac{2}{3}t_1-\frac{1}{3}t_2)}, \\
f_{\vsg({\bf y})}(\vt)&=\frac{Z_S^{X, \Phi}(\vsg({\bf y})+\ii \vt)}
{Z_S^{X, \Phi}(\vsg({\bf y}))}=\beta_1\e^{\ii(-\frac{1}{3}t_1-\frac{2}{3}t_2)}
+\beta_2\e^{\ii(-\frac{1}{3}t_1+\frac{1}{3}t_2)}
+\beta_3\e^{\ii(\frac{2}{3}t_1+\frac{1}{3}t_2)},
\end{aligned}
$$
for $\vt=(t_1, t_2) \in \mathbb{R}^2$. 
These characteristic functions corresponds to the $\mathbb{R}^2$-valued random variables 
$\mathcal{Y}_{{\bf x}}$ and $\mathcal{Y}_{{\bf y}}$ given by
$$
{\bf P}\Big(\mathcal{Y}_{{\bf x}}=\Big(\frac{1}{3}, \frac{2}{3}\Big)\Big)
=\alpha_1, \quad 
{\bf P}\Big(\mathcal{Y}_{{\bf x}}=\Big(\frac{1}{3}, -\frac{1}{3}\Big)\Big)
=\alpha_2, \quad 
{\bf P}\Big(\mathcal{Y}_{{\bf x}}=\Big(-\frac{2}{3}, -\frac{1}{3}\Big)\Big)
=\alpha_3 
$$
and
$$
{\bf P}\Big(\mathcal{Y}_{{\bf y}}=\Big(-\frac{1}{3}, -\frac{2}{3}\Big)\Big)
=\beta_1, \quad 
{\bf P}\Big(\mathcal{Y}_{{\bf y}}=\Big(-\frac{1}{3}, \frac{1}{3}\Big)\Big)
=\beta_2, \quad 
{\bf P}\Big(\mathcal{Y}_{{\bf y}}=\Big(\frac{2}{3}, \frac{1}{3}\Big)\Big)
=\beta_3,
$$
respectively, which means that the random walk $\{\mathcal{W}_n\}_{n=0}^\infty$ on $X$ generated by the multidimensional 
Shintani zeta functions is the usual nearest-neighbor one and also 
the simple one if $\alpha_1=\alpha_2=\alpha_3=\beta_1=\beta_2=\beta_3=1/3$. 

\end{ex}








\end{document}